\documentclass[11pt]{article}
\usepackage{amssymb,amsmath,amsthm}
\usepackage{color}
\usepackage{upgreek}
\usepackage{stmaryrd}
\usepackage{hyperref}
\newtheorem{definition}{Definition}

\newtheorem{proposition}{Proposition}
\newtheorem{lemma}{Lemma}
\newtheorem{corollary}{Corollary}
\newtheorem{theorem}{Theorem}
\newtheorem*{theorem1bis}{Theorem 1bis}

\newcommand{\Z}{\mathbb{Z}}

\renewcommand{\mu}{\upmu}

\def\bd{\partial}

\def\rest{\hskip 1pt{\hbox to 10.8pt{\hfill
\vrule height 7pt width 0.4pt depth 0pt\hbox{\vrule height 0.4pt
width 7.6pt depth 0pt}\hfill}}}

\newcommand{\eps}{\varepsilon}

\def\beq{\begin{equation}}
\def\eeq{\end{equation}}

\def\R{{\mathbb R}}
\def\N{{\mathbb N}}

\def\bd{\partial}

\def\QED{\hbox{${\vcenter{\vbox{
   \hrule height 0.4pt\hbox{\vrule width 0.4pt height 6pt
   \kern5pt\vrule width 0.4pt}\hrule height 0.4pt}}}$}\vspace{7pt}}

\begin{document}

\author{Fabrice BETHUEL\thanks{UPMC-Paris6, UMR 7598 LJLL, Paris, F-75005 France, 
}, Didier SMETS\thanks{UPMC-Paris6, UMR 7598 LJLL, Paris, F-75005 France} and Giandomenico ORLANDI\thanks{Universit\`a di Verona, Dip. Informatica, 15 strada le Grazie, I-37134 Verona}}
\title{Slow motion  for gradient systems with equal depth multiple-well potentials}
\date{}
\maketitle

\begin{abstract}
 For scalar reaction-diffusion in one space dimension,  it is known for a long time that fronts move with an exponentially small speed   for   potentials with several distinct minimizers.  The purpose of this paper is to provide a similar   result  in  the case  of systems. Our method  relies on a careful study of the evolution of localized energy. This approach has the advantage to relax the  preparedness assumptions on the initial datum.
\end{abstract}

\bigskip
\noindent
\section{Introduction}

\subsection{Potentials with wells of equal depth}

 The purpose of this paper is to investigate  the behavior of  solutions $v$ of the reaction-diffusion equation of gradient type
 \begin{equation}
 \label{glpara}
  v_t- v_{xx}=-\nabla V(v).
  \end{equation}
The function  $v$  denotes here  a function of the space variable $x \in \R$ and the time variable $t\geq 0$ and takes values in some euclidean space $\R^k$, so that \eqref{glpara} is a system of $k$ scalar partial differential equations. 
   Equation \eqref{glpara}   actually corresponds to the $L^2$ gradient-flow    of the  energy functional  $\mathcal{E}$ which is defined for a function $ u:\R \mapsto \R^k$ by the formula
    \begin{equation}
 \label{glfunctional}
 \mathcal{E}(u)= \int_\R e(u)= \int_{\R} \frac{\vert \dot u \vert ^2}{2}+V(u).
 \end{equation}
   The function $V$, usually termed the potential, is assumed to be  a smooth  function from $\R^k$ to  $\R$, tending to infinity at infinity, so that it is bounded from below.    
   
 Simple solutions to equation \eqref{glpara}  are provided by the stationary ones,  that is time-independent solutions of the form $v(x, t)=u(x)$, where the profile $u:\R\mapsto \R^k$ is a solution of the ordinary differential equation
       \begin{equation}
       \label{ordi}
      - u_{xx}=-\nabla V(u).
      \end{equation}
Among those solutions, the simplest ones are the constant functions $v(x,t)=\sigma$, where $\sigma$ is a critical point of  the potential $V$,  for instance  a minimizer. Another interesting type of solutions to \eqref{ordi} is provided by those tending, as $x \to \pm \infty$, to critical points of the potential $V$: in this case conservation of energy for \eqref{ordi} implies that  $V(u(+\infty))=V(u(-\infty))$. The central assumption on the potential $V$ in this  paper is that it possesses a finite number of, and at least two, distinct minimizers.  A canonical example  in the  scalar case $k=1$ is given by the function
 \begin{equation}
 \label{glexemple}
  V(u)=\frac{(1-u^2)^2}{4}, 
  \end{equation}  whose minimizers are $+1$ and $-1$.
  
  This paper is devoted to the analysis  of   the evolution in time of  initial data which connect two distinct minimizers of the potential.  Such maps $u$, from  $\R$ to $\R^k$, whose limits at $\pm  \infty$ are distinct minimizers of $V$  are usually termed {\bf  fronts}. If they are moreover solution to the ordinary differential equation \eqref{ordi}, we will call them stationary fronts, so that a stationary front is a heretoclinic\footnote{Actually homoclinic solutions, whenever they exist, could be considered as well.} solution to \eqref{ordi}.
 For instance, in example \eqref{glexemple}, stationary fronts are necessarily of the form
 \begin{equation}
 \label{kink}
  u(x) = w_{\pm}(x-c)=\pm \tanh \left( \frac{x-c}{\sqrt{2} }\right),  
\end{equation}
for some $c\in \R.$ In that case, $w_{+}$ (resp. $w_{-}$), which is often referred to as the kink (resp.  anti-kink) solution,  connects $-1$ to $+1$  (resp. $+1$ to $-1$).
  The dynamics of fronts and their  eventual convergence to stationary fronts, which are attractors of the dynamics,   is actually  a central topic in the study  of reaction-diffusion equations of  gradient-type: in our short historical section below, we will review some of the works related to this question. \par
 
 To be more  specific,  our assumptions on  the potential $V$ can be formulated as follows. We assume that $V$ is smooth and satisfies the conditions:
  \smallskip 
   $$ \inf V=0 \hbox{ and the set of minimizers}  \  \Sigma\equiv \{ y \in \R^k, V(y)=0 \}\  \leqno{(\text{H}_1)}$$
 is a finite set, with at least two distinct elements, that is
\begin{equation}
\label{sigma}
 \Sigma=\{\upsigma_1, ..., \upsigma_q\},\ q\geq 2, \ \upsigma_i \in \R^k, \ \forall i=1,...,q.
 \end{equation}
  \smallskip
       
 \noindent
${(\text{H}_2)}$  The matrix $\nabla^2V(\upsigma_i)$ is positive definite at each point $\upsigma_i$ of $\Sigma$, in other words, if $\lambda_i^-$ denotes its smallest eigenvalue, then  $\lambda_i^->0$. We denote by $\lambda_i^+$ its largest eigenvalue. 

\smallskip
\noindent
  ${(\text{H}_3)}$ There exists  constants  $\alpha_0>0$ and  $R_0>0$  such that  
   $$y\cdot\nabla V( y )\geq \alpha_0 \vert y \vert ^2, \ \hbox {if }  \vert y \vert >R_0.$$

\smallskip

  A potential $V$ which fulfills conditions ${(\text{H}_1)}$, ${(\text{H}_2)}$ and ${(\text{H}_3)}$  will be termed  throughout a non-degenerate multiple-well potential with equal depths.  A  canonical example is given by \eqref{glexemple},  for which $\Sigma=\{+1,-1\}.$

\medskip

 The main assumption  in this paper  on the initial datum $v^0(\cdot)=v(\cdot, 0)$ is that its energy is finite. More precisely, given an arbitrary constant 
 $M_0>0$, we assume that
 $$
\mathcal E(v^0)\leq M_0<+\infty.  \leqno\text{$(\text{H}_0)$}
 $$
      In particular, in view of the classical energy identity
\begin{equation}
\label{energyidentity}
\mathcal E(v(\cdot,T_2))+\int_{T_1}^{T_2} \int_{\R} \left| {\bd
v \over \bd t}
 \right|^2(x,t)dx\,dt = \mathcal E(v(\cdot,T_1)) \, \quad \forall\, 0\le
 T_1\le T_2\, , 
\end{equation}
 we have, $\forall t>0$,
 \begin{equation}
 \label{touttemp}
 \mathcal E\left(v(\cdot, t)\right )\leq  M_0.
 \end{equation}
 
 This implies in particular that for every  given $t\geq 0$, we have  $V(v(x, t))\rightarrow 0$ as $\vert x\vert \rightarrow \infty$. It is then quite straightforward to deduce from assumption ${(\text{H}_0)}$, ${(\text{H}_1)}$, ${(\text{H}_2)}$ as well as the energy identity  \eqref{energyidentity}, that $v(x, t)\rightarrow \upsigma_{\pm}$ as $x\rightarrow \pm \infty$ , where $\upsigma_{\pm} \in \Sigma$ does not depend on $t$.  In other words our assumptions imply that the map $v(\cdot, t)$ is a front  for all times $t>0$ if $\upsigma_+\neq\upsigma_-$.
 
 \medskip

    \noindent
    \subsection{ Front sets}
     One of our aims is to localize the evolution in time of the region where the function $v(\cdot, t)$ jumps from one minimizer of $V$ to a second one. This will allow us to follow the evolution of the front.  To that purpose,
 we fix $\upmu_0>0$ sufficiently small so that, for $i=1, \ldots, q$, we have
 $$B(\upsigma_i, \upmu_0)\cap B(\upsigma_j, \upmu_0) = \emptyset $$
 for all $i\neq j$ in $\{1,\cdots,q\}$
 and
 \beq\label{eq:conv}
 \frac{1}{2}\lambda_i^-{\rm Id} \leq  
 \nabla^2 V (y) \leq 2\lambda_i^+{\rm Id} 
     \eeq
for all $i\in \{1,\cdots,q\}$ and $y \in B(\upsigma_i,\upmu_0).$
       We then define, for a map $u:\R\mapsto \R^k$,  the set 
    \begin{equation}
    \label{frontset}
     \mathcal D (u)\equiv\{x\in \R,\ \hbox{dist}(u(x),\Sigma) \geq \upmu_0 \}.
     \end{equation}
     In the context of equation \eqref{glpara} we set moreover
     \begin{equation}
     \mathcal D(t)=\mathcal D(v(\cdot, t)).
     \end{equation}
     The evolution of the set  $\mathcal D(t)$ is the main focus of our paper.

\medskip

   For a given map $u$, $\mathcal D(u)$ is related to the set where the energy of $u$ concentrates, in view of the following: 
  \begin{lemma} 
  \label{clearingout} There exists a constant $\eta_0>0$,  depending only $\upmu_0>0$ and $V$, such that, if $I$  is an  interval of $\R$ of length $\vert I \vert \geq 1$,  and $u$ is a $\R^k$-valued function on $I$ satisfying
  \begin{equation}
  \label{riri}
  \int_I  e(u)\leq \eta_0,
  \end{equation}
    then
   \begin{equation}
   \label{gavroche}
     {\rm dist} (u(x), \Sigma) < \upmu_0  \qquad \text{for all } x\in I,
   \end{equation}
   or equivalently 
\begin{equation}
   \label{gavroche2}
     \mathcal D(u) \cap I = \emptyset.
   \end{equation}

  \end{lemma}
  This kind of result is usually called a {\it clearing-out Lemma} in the literature. It shows that if the energy is sufficiently small in some place, then there are no front located there, or equivalently that where fronts are present, energy needs to concentrate\footnote{The converse is of course not true in general for arbitrary maps, think of small oscillations.}. Therefore, fronts  are among energy concentration intervals, and energy is a good object to track fronts. 
  We will explicitly assume in the sequel $M_0\ge \eta_0$.

An immediate consequence of Lemma \ref{clearingout} yields:

\begin{corollary} 
 \label{interface} 
  Assume  that the map $u$  satisfies $\mathcal E(u)\leq M_0$. There exists $\ell$ points $x_1,...,x_\ell $ in $\mathcal D(u)$ such that 
 \begin{equation}
 \mathcal D(u) \subset \cup_{i=1}^{\ell} [x_i- 1,x_i +1],
 \end{equation}
with a bound $\ell \leq \frac{M_0}{\eta_0}$ on the number of points. 
\end{corollary}

\medskip
\noindent
\subsection{Slow motion of concentration sets}   
The first and main result of this paper is as follows. Assume $M_0\ge\eta_0$, and set
$$
\alpha_0 = 32\frac{M_0}{\eta_0}.
$$
\begin{theorem}
\label{maintheo}
Assume that the potential $V$  satisfies assumptions  
 ${(\text{H}_1)}$, ${(\text{H}_2)}$ and ${(\text{H}_3)}$, that the initial datum $v_0$ satisfies the energy bound ${(\text{H}_0)}.$
There exists a constant $K_0>0$ depending only on the potential $V$ and on $M_0$ such that if $R\geq \alpha_0$, 
\begin{equation}\label{eq:main}
\mathcal D (t) \subset \mathcal D(0) + [-R,R]
\end{equation}
provided 
$$ 0\leq t \leq  \left(\frac{R}{K_0}\right)^2 \exp\left(\frac{R}{K_0}\right).$$ 
\end{theorem}

Theorem \ref{maintheo} expresses the fact the  motion of the front set is slow when considered at sufficiently large scale $R.$ Indeed, its maximal average speed should not  exceed 
$$c(R)=K_0^2R^{-1}\exp\left(-\frac{R}{K_0}\right).$$ 
For large $R$ this speed is exponentially small. As a matter of fact, our proofs rely on an asymptotic expansion for large values of $R$. 
      
Since $R$ represents the typical length scale which is considered, if one wishes to work on fixed domains, it is sometimes interesting to introduce a small parameter $\eps>0$, and consider the more general form of the equation \eqref{glpara} given by
    \begin{equation}
 \label{glepsilon}
 {\partial_t v_\eps}- \partial_{xx}{v_\eps}=-\frac{1}{\eps^2}\nabla V(v_\eps).
  \end{equation}
Notice that if $v$ is a solution to \eqref{glpara},  then the map $v_\eps$ given by 
\begin{equation}
\label{rescaling}
v_\eps(x, t)=v(\eps x, \eps^2t)
 \end{equation}
is a solution to \eqref{glepsilon}. In this setting the statement of Theorem \ref{maintheo} may easily be translated if we replace equation \eqref{glpara} by equation \eqref{glepsilon} and the map $v$ by the map $v_\eps.$ 
The energy density $e_\eps(u)$ and the energy functional $\mathcal E_\eps(u)$ are  given respectively  by 
\begin{equation}
\label{eeps}
e_\eps(u)=\frac{\eps}{2} \dot u ^2+ \frac{1}{\eps} V(u)\quad  {\rm and \  }\quad 
\mathcal E_\eps(u)=\int _\R e_\eps(u).
\end{equation}
The energy bound ${(\text{H}_0)}$  is translated into
 $$
\mathcal E_\eps(v^0)\leq M_0<+\infty.  \leqno\text{$(\text{H}_0^{\, \eps})$}
 $$ We may define accordingly  the front set $ \mathcal D(t)=\mathcal D(v_\eps(\cdot, t))$,  so that if the initial datum satisfies the initial bound  ${(\text{H}_0)^{\,\eps}}$, this set is of size  of order $\eps$, in view of Corollary \ref{interface} and \eqref{rescaling}, and shrinks, as $\eps$ converges to zero, to a finite set, which is sometimes termed the defect set.
Considering $R>0$ fixed but letting $\eps>0$ vary, the  statement of Theorem \ref{maintheo} turns  then out to be  equivalent to the following statement.
 
\begin{theorem1bis}
 Assume that the potential $V$  satisfies assumptions  
 ${(\text{H}_1)}$, ${(\text{H}_2)}$ and ${(\text{H}_3)}$, let $\eps>0$  be given and consider a solution $v_\eps$ to \eqref{glepsilon}. Assume  that the initial datum $v^0_\eps(\cdot)=v_\eps(\cdot, 0)$ satisfies the energy bound ${(\text{H}_0^{\, \eps})}.$ There exists a constant $K_0>0$ depending only on the potential $V$ and on $M_0$ such that if $R\geq \alpha_0\eps,$ then
\begin{equation}\label{eq:main2}
\mathcal D (t) \subset \mathcal D(0) + [-R,R]
\end{equation}
provided 
$$ 0\leq t \leq  \left(\frac{R}{K_0}\right)^2 \exp\left(\frac{R}{K_0\eps}\right).$$ 
\end{theorem1bis}
As we will  recall in our short historical survey on the topic below,   slow motion of fronts  has a long history in the mathematical literature,  and it is  mainly  described  considering the form \eqref{glpara} and assuming $\eps$ is asymptotically small.

\medskip

Combining  Theorem 1bis with dissipation estimates off the front set (see Section \ref{off} below), we obtain

\begin{theorem}
	\label{thm:pouf}
	Assume that $V$ and $v_\eps$ are as in Theorem 1bis. There exist a constant $K_1>0$ depending only on $V$ such that if $x_0\in \R$ and $R\geq \alpha_0\eps$ satisfy 
	$$
[x_0-2R,x_0+2R] \cap \mathcal D(0) = \emptyset,
	$$
	then
	\begin{equation}\label{eq:plouf1}
		\eps^3|\partial_t v_\eps|^2 + e_\eps(v_\eps) \leq K_1 M_0\eps^{-1} \left[ \exp\left(-\frac{t}{K_1\eps^2}\right) + \frac{t}{R^2}\exp\left(-\frac{R}{K_1\eps}\right)  \right],
	\end{equation}
	pointwise on $[x_0-\frac12R,x_0+\frac12R]\times [\eps^2, (\frac{R}{K_0})^2\exp(\frac{R}{K_0\eps})].$
\end{theorem}

Using a Gronwall type argument, we may then prove that the flow drives the solution close to a chain of stationary fronts.  More precisely, we have
 
 \begin{theorem}
\label{theglue}   Assume that $V$ and $v_\eps$ are as in Theorem 1bis. Given $R\geq \alpha_0 \eps$ there exists a relaxation time
$$
0 \leq T \leq  \left(\frac{R}{K_0}\right)^2 \exp\left(\frac{R}{K_0\eps}\right)
$$
at which $v_\eps(\cdot,T)$ possesses the following structure: There exist $K_2>0$ depending only on $V$ and $M_0$, a length scale of order $R$
$$
2^{-4\frac{M_0}{\eta_0}}\frac{R}{K_2} \leq r \leq \frac{R}{K_2}, 
$$
a collection of points $\{a_j\}_{j\in J}$ in $\R$, and a corresponding collection of functions $\{U_j\}_{j\in J}$ defined on $[-\frac{r}{\eps},\frac{r}{\eps}]$ with values into $\R^k$ such that
\begin{enumerate}
	\item $0\leq \sharp J  \leq \frac{M_0}{\eta_0},$
	\item $a_j \in \mathcal D(T)$  \qquad $\forall\ j\in J,$
	\item ${\rm dist}\left(a_j,\mathcal D(0)\right)\leq R,$\qquad $\forall\ j\in J,$
	\item ${\rm dist}(a_i,a_j) > 4r$\qquad $\forall\ i\neq j\in J,$
	\item Each $U_j$ is a solution to the stationary equation \eqref{ordi} with zero discrepancy:
		$$
		-\partial_{xx}U_j = -\nabla V(U_j) , \qquad \xi(U_j)=\frac{|\partial_x U_j|^2}{2}-V(U_j) =0,
		$$
	\item We have the estimate
		$$
\big\| v_\eps(\cdot, T)- U_j\left(\tfrac{\cdot-a_j}{\eps}\right)\big \| + \eps \big\|\partial_x\left( v_\eps(\cdot, T)- U_j\left(\tfrac{\cdot-a_j}{\eps}\right)\right) \big \| \leq K_2 \exp\left(-\frac{r}{K_2\eps}\right), 
$$
in $L^\infty([a_j-r,a_j+r]),$ for each $j\in J$,
\item
	If $I$ is an interval disjoint from $\cup_{j\in J}[a_j-r,a_j+r]$,  we have 
$$
\| v_\eps(., T)- \sigma_i\| + \eps\| \partial_x v_\eps(\cdot,T) \| \leq K_2 \exp\left(-\frac{r}{K_2\eps}\right),
$$
in $L^\infty(I),$ for some $\sigma_i \in \Sigma.$
\end{enumerate}
\end{theorem}

Notice that the functions $U_j$ are defined on the interval 
$[-\frac{r}{\eps},\frac{r}{\eps}]$, which grows as $\eps$ tends to $0$ if $R$ is kept fixed, to cover the whole of $\R$. In particular if one considers a family of solutions $(v_{\eps})_{0<\eps<1}$ to  \eqref{glepsilon}  satisfying the energy bound {$(\text{H}_0^{\eps})$ and the corresponding family $U_j\equiv U_j^\eps$ obtained thanks to Theorem \ref{theglue}, then a straightforward compactness argument shows that, up to a subsequence $\eps_n \to 0$, we have the convergence
 \begin{equation}
 \label{frontignan} 
 U_j^{\eps_n} \to U_j^0, {\rm  \ as \ } n \to \infty, 
 \end{equation} 
 in $C^k(K)$ for any $k \in \R$ and any compact interval $K$ of $\R$, where the limiting map $U_j^0$  is defined on the whole of $\R$ and is a stationary front, that is a finite energy solution to \eqref{ordi}.  Loosely speaking\footnote{Loosely in particular because the $U_j$'s are {\it not} globally defined !},  one may rephrase Theorem \ref{theglue} stating  that, after some suitable time, the solution enters  an $O\left(\exp(-\frac{R}{\eps})\right)$ neighborhood of glued together stationary fronts.
  
   The next step in the analysis would be to derive a precise motion law for the fronts, or more precisely, in view of Theorem \ref{theglue}, of the points $a_j$. It is in particular of interest to determine whether  the interactions between them is repulsive or attractive. Notice however that the points $a_j$  are only defined so far up to an $O(R)$ term. This introduces an additional difficulty which is easily removed in the scalar case imposing some preferred value for $U_i(0)$, in the vectorial case the situation is more subtle. We are inclined  to believe that such results would involve more restrictive assumptions on the potential $V$ than the ones which we have used so far, which are rather mild and involve only its behavior near its zeroes. In particular, more should be known or required regarding the zoology of finite energy stationary front (concerning for instance their spectral properties).
  
   In the scalar case, it is a standard exercise to integrate equation \eqref{ordi} and to determine the set of finite energy solutions (for instance, the solution is given by formula \eqref{kink} in the case the potential is \eqref{glexemple}).  Furthermore in that case, fixing the discrepancy to zero as it is done in statement $5$ of Theorem \ref{theglue} insures that all the solutions $U_j$ are true stationary fronts\footnote{In particular, for the potential given by \eqref{glexemple}, Theorem \ref{theglue} allows to recover the convergence result of Fife and McLeod \cite{fifemac} without any reference to the maximum principle (this was already achieved in Gallay and Risler \cite{gallayrisler}), and with a quantitative error estimate (we do not rely on a compactness argument).} (that is they coincide with finite energy globally defined solutions of \eqref{ordi}).   In a forthcoming paper, we will show how these information combined with the local energy identity which is the central tool of the present paper allows to recover the motion law in the Allen-Cahn case (see the historical notes below), where  the fronts attract, and also can be extended to three wells or more, where the interactions may be attractive or repulsive, depending on the nature of the fronts.

An ultimate goal would be of course to obtain similar results in the case of systems, when  appropriate assumptions are made on the potential.  In that case, we believe that the interaction  between two  fronts is governed by the behavior as $x \to \pm\infty$ of the  corresponding finite energy stationary fronts. In order to state a conjecture which specifies the magnitude of the interaction, consider   first a given finite energy stationary solution $U$ to \eqref{ordi}, and set $\sigma^\pm= {\underset{x\to \pm \infty }\lim } U(x)$. In view of the form of the potential near its zeroes, it  can be shown that the solution converges to 
  $\upsigma^+$  (resp. $\upsigma^-$) with an asymptotic direction, that is  that the following limits exist
 $$\upomega^+= {\underset{x\to + \infty }\lim}  \frac{U(x)-\upsigma^+}{\vert U(x)-\upsigma^+\vert} \ \ \left ({\rm  resp. \ }\upomega^-= {\underset{x\to -\infty }\lim}  \frac{U(x)-\upsigma^-}{\vert U(x)-\upsigma^-\vert}\right), $$
  so that $\vert \upomega^+\vert=\vert \upomega^-\vert=1.$
  Next, we go back to Theorem \ref{theglue} and consider two  consecutive ``fronts'', say $U_1$ and $U_2$ given by its statement. We have in particular $\upsigma^+(U_1)=\upsigma^-(U_2)\equiv \upsigma.$ Let $\lambda>0$ denote the smallest  eigenvalue of the  matrix $\nabla^2V(\upsigma)$, let  $F$ denote the corresponding eigenspace, and $\Pi$ the projection onto $F$.  Our conjecture is that the interaction between the two fronts is given by 
  $$ \left( C\langle \Pi \upomega_1^+,  \Pi \upomega_2^-\rangle_{\R^k} + o(1)\right) \exp\left( -\frac{\lambda}{\eps}d\right),\qquad \text{as }\eps \to 0, $$
     where the constant $C>0$ depends on the profiles $U_1$ and  $U_2$ and $d$ denotes the distance between the two fronts.\footnote{As explained above, defining the distance between two fronts supposes first to localise them by fixing an anchor point for each of them, the constant $C$ depends of course on the definition of those anchors.}

\medskip

\noindent
{\bf A short  historical  review}.  General reaction-diffusion  of gradient type \eqref{glepsilon}  have been  widely  introduced  and used as models  in numerous branches of science, for instance  in  physics, chemistry (in particular combustion theory), or  biology,   among many others. The theory of fronts has received  extensive mathematical study, in particular the theory is highly developed in  the scalar case.  For instance, existence and uniqueness (up to translation) of traveling fronts\footnote{that is,  solutions with a constant profile which are translated with constant speed } has been established in the scalar case. It has also been shown that for  arbitrary initial data connecting the local minimizers, as time tends to infinity solutions converge towards such fronts (see e.g. the seminal work of Fife and McLeod \cite{fifemac}). More recently, part of the analysis of convergence towards traveling fronts as time tends to infinity has been extended to the case of systems by Risler \cite{risler, risler1, risler2} and Gallay and Risler\cite{gallayrisler}. The methods used in these works rely on energy estimates and compactness arguments. 

When a front connects two local minimizers with the same potential energy, it is stationary, and therefore is a solution of \eqref{ordi}. As proved by Risler, as time tends to infinity solutions for ``arbitrary'' initial data eventually converge to slowly repulsing chains of such stationary fronts.
  The next step of the analysis, when the initial data possesses several fronts, is to follow carefully  the evolution of  the various fronts from the initial time, and not only asymptotically, and possibly to estimate their speed. On a heuristic level, the speed of the fronts can be seen as the effect of the (small) interaction between them: This interaction might be attractive or repulsive depending on the nature of the potential.
  
For scalar two-wells potentials,  which are often referred to as Allen-Cahn potentials, this program was first completed in the celebrated works of Carr and Pego, and Fusco and Hale  \cite{carpego, carpego2, fuschale}, which provided  the first  rigorous mathematical derivation of very slow motion, and even derived a precise motion law for the evolution of fronts.  In their result, the initial data  is very constrained, since it is supposed to be close to optimally glued together fronts. Their method  relies on a careful analysis of the motion near these special solutions through  a thorough study of the linearized operator near the stationary solution.  The fact that the kernel for such solutions on the line contains only the  space derivative of the solutions is crucial there (the proof of  this latest statement provided in \cite{carpego} relies heavily on the fact that the solution is scalar). This type of spectral methods were later applied successfully on related  problems, for instance the Cahn-Hilliard equation (see e.g. \cite{alikakos2, alikakos1}). Other interesting papers based on that kind of ideas (sometimes termed the invariant manifold method or geometric method) are \cite{chenxy, chengene, fuschale, pinto, ei}. In particular in \cite{ei}, Ei was able to handle the interaction of a kink and an anti-kink in the vectorial case. Keller, Rubinstein and Sternberg \cite{kerustern} made related contributions in a similar direction. 
    
At least two other methods have been applied successfully in the scalar case. Firstly, the method of sub-solutions and super-solutions turns out to be extremely powerful and allowed to handle larger classes of initial data (see e.g. \cite{fifemac, chengene}). There is little hope however to extend this method to systems, since comparison principles do not hold in general for systems. Another direction is given by the global energy approach due to Bronsard and Kohn \cite{BrKo}. Using the energy  identity \eqref{energyidentity}, they were able to prove that for initial data sufficiently close to glued front solutions, the fronts have a speed slower than $O(\eps^k)$, for every $k \in \N$. The closeness to the glued solutions is expressed in their paper in terms of energy estimates and assumptions which are reminiscent of concepts of Gamma-convergence.  Grant  \cite{grant} improved the method to obtain an exponentially small upper bound of the form $O\left(\exp(-\frac{c}{\eps})\right)$ imposing however stronger conditions on the initial data. The method was extended to functionals  with higher order derivatives in \cite{vandervorst}. Finally, in \cite{ottorez} through a more abstract setting of the problem, Otto and Reznikoff were able to recover some of the results in \cite{chengene} through global energy methods.

At this stage it is worthwhile to emphasize that all mentioned results use more or less the properties of the solutions to \eqref{ordi} (in the case of Allen-Cahn, these are unique up to translations and sign change).\\

   The aim of our present paper is to extend the analysis to the case of systems, to relax the assumptions on the preparedness of the initial data, and also to handle possible annihilation of fronts. Our approach bears many  analogies with the  energy method of Bronsard, Kohn and Grant, however, instead of using the global energy identity we use a local version of it, which is combined with parabolic estimates away from the fronts. Our  ideas are partially borrowed from our earlier work on the motion of vortices in the two-dimensional parabolic Ginzburg-Landau equation \cite{BOS4,BOS5} as well  as  from earlier works on the topic by Lin, Jerrard and  Soner \cite{Li, JeSo}. A new  technical difficulty  which occurs in  the present paper is that we are able to handle, for the evolution of localized energy density, test functions which are affine (and therefore change sign) near the front, whereas (positive) quadratic test functions were extremely useful in the context of the Ginzburg-Landau vortices. In contrast with the results obtained so far in the scalar case, our results {\it do not}  rely on the properties nor the existence of solutions to \eqref{ordi} (as a matter of fact, the number and the properties of solutions to \eqref{ordi} might be much more involved in the multi-dimensional case than in the scalar case). Roughly speaking the heart of our method is precisely to avoid regions where the solution becomes close to solution to \eqref{ordi}.\\

  \subsection{ Elements in the proofs}    
    For the proofs of the main results, we will work with the parameter $\eps$ (equation \eqref{glpara} being a special case for the value $\eps=1$). The most difficult part corresponds to the case where $\eps$ is small with respect to $R.$ 
    
      In order to analyze the  evolution in time  of the concentration sets of the energy, we invoke the localized version of \eqref{energyidentity}, which  writes, for   a smooth  test function $\chi$  with compact support in $\R$
   \begin{equation}
\label{eq:loca1}
\frac{d}{dt}\int_{\R\times\{t\}}  \hspace{-10pt}e_\eps(v_\eps) \chi(x)\,dx = -
\int_{\R\times\{t\}} \hspace{-10pt} \eps|\partial_t v_\eps|^2\chi(x)\,dx -
\int_{\R\times\{t\}} \hspace{-10pt}\eps\partial_t v_\eps \nabla v_\eps\cdot \nabla \chi \, dx.
\end{equation}
A few integration by parts yield  the classical formula \eqref{localizedenergy} below.

   \begin{lemma} 
   \label{alpha}Let $\chi$ be a smooth function with compact support on $\R$. Then, we have the identity
\begin{equation}
\label{localizedenergy}
\frac{d}{dt}\int_{\R} \chi(x)\,e_\eps(v_\eps)dx= - \int_{\R\times\{t\}}
\eps\chi(x) |\partial_t
v_\eps|^2dx+\mathcal{F}_S(t,\chi,v_\eps),
\end{equation}
where, the term $\mathcal{F}_S$, is given by
 \begin{equation}
 \begin{split}
\mathcal{F}_S(t,\chi,v_\eps)= \int_{\R \times\{t\}} \left(    
\left[ \eps \frac{\dot v_\eps^2}{2}-\frac{V(v_\eps)}{\eps}\right ]\ddot \chi
\right)\,dx.
\end{split}
\end{equation}
\end{lemma}
\medskip

 The first term on the right hand side of identity \eqref{localizedenergy}  stands for local dissipation, whereas the second might be interpreted as  a flux. The quantity
 \begin{equation}
 \label{discrepance}
  \xi_\eps \equiv[\eps \frac{\dot v_\eps^2}{2}-\frac{V(v_\eps)}{\eps}]
  \end{equation}
  is   sometimes   referred to as {\bf the discrepancy term}  in the literature. For  solutions of the ordinary differential equation 
    \begin{equation}
    \label{glordi}
    -{u_\eps}_{xx}+\frac{1}{\eps^2}\nabla V(u_\eps)=0  \qquad  { \rm on} \  I,
    \end{equation}
for some interval $I$, the discrepancy is constant, and in particular it vanishes if the interval is the whole of $\R$ and the solution is a heteroclinic between two of the wells. A general idea, which is underlying our analysis, is the  fact that front sets in the parabolic equations relax quickly to solutions to \eqref{glordi}, so that we may take advantage of some properties of the ode \eqref{glordi}, in particular in order to estimate the driving force  $\mathcal{F}_S$.
 
 \medskip
 
  We will make use of  formula \eqref{localizedenergy} for  a rather specific  choice of test functions $\chi$, namely  test functions which are affine  near the defect set, so that  $\ddot\chi$ vanishes there, and one has merely, in view of Lemma \ref{alpha}, to estimate the discrepancy term off the defect set. An important step in our  proofs is to show that if the solution is locally close to one of the minimizers $\upsigma_i$ at a given time, it remains so for a short while, and the equation has strong smoothing properties on the corresponding space-time region.

We have
   
\begin{proposition}
\label{regularisation}
Assume that $v_\eps$ is a solution to \eqref{glepsilon} verifying assumption $(H_0)$. Let $t\geq 0$, $x\in \R$ and $r\geq \alpha_0 \eps$ be such that 
\begin{equation}
\mathcal D(t) \cap [x-r,x+r] = \emptyset.
\end{equation}
Then, 
\begin{equation}\label{eq:kloden}
\mathcal D(s) \cap [x-r/2,x+r/2] = \emptyset
\end{equation}
for every $t\leq s \leq t+ \alpha_0^{-3}r^2.$
Moreover, for such $s$, 
\begin{equation}\label{eq:mouloud3}
\int_{x-r/2}^{x+r/2} e_\eps(v_\eps(y,s))\, dy \leq M_0(1+4\frac{\lambda_i^+}{\lambda_i^-})\left[ \exp(-\frac{\lambda_i^-}{\eps^2}(s-t)) + 2^{14} \frac{s-t}{r^2} \exp(-\sqrt{\frac{\lambda_i^-}{2}}\frac{1}{12}\frac{r}{\eps})\right],
\end{equation} 
where $\lambda_i^-$ and $\lambda_i^+$ are defined in $(H_2)$ for the corresponding potential well.
\end{proposition}
  
Finally, another important tool in our proofs in an elementary covering argument, which we introduced 
in an earlier work on the motion of Ginzburg-Landau vortices  and that  we will recall   in  subsection \ref{kappa}. Roughly speaking,  given a  small number $\kappa$, this covering results  states that given $\ell$ points in a metric space and a typical length $\delta$,  one may  find a length scale $\tilde \delta$ of the same order as $\delta$, such that all the points are included in balls of radius $\kappa \tilde \delta$,  and such that  mutual distances between the balls is at least $\kappa^{-1} \tilde \delta$ (the points are said to be $\kappa$-confined inside these balls).

The proof of Theorem 1bis
is then a combination of  applications of   Lemma \ref{alpha}, Proposition \ref{regularisation},  the clearing-out Lemma \ref{clearingout}, and   the $\kappa$-confinement Lemma.   It is  based on an iteration argument, which will be completed in a finite, bounded by above,  number  of steps.  The main parameter which is involved in the iteration process is $\delta$,  which appears in the $\kappa$- confinement technique,  and  which we let it shrink. More precisely, we set  $\delta_{n+1}=\delta_n \kappa^{\ell_0}$,  and each step on the iteration is defined thank to a suitable stopping time.

\numberwithin{theorem}{section} \numberwithin{lemma}{section}
\numberwithin{proposition}{section} \numberwithin{remark}{section}
\numberwithin{corollary}{section}
\numberwithin{equation}{section}

\section{Properties of the  front set}\label{sect:front}
We describe in this section some classical and mostly elementary properties of the front set. More precisely, we present the proofs of Lemma \ref{clearingout} and Corollary \ref{interface}.

 \medskip
\noindent
{\bf Proof of Lemma \ref{clearingout}.}
 Let $a\in I$ be such that  $[a, a+1]\subset I$. Since by assumption $\vert I \vert \geq 1$, the set of such points is non-empty. In view of the definition of the energy, $V(u)\leq e(u)$, and therefore if the map $u$ satisfies \eqref{riri} then 
$$\int_{[a, a+1]}V(u(x))dx  \leq \int_{I}V(u(x)) dx\leq \eta_0.$$
We deduce  from the mean-value theorem that there exists some point $s_0 \in [a, a+1]$ such that
\beq\label{eq:trou}V(u(s_0))\leq \eta_0.\eeq
We first choose the constant $\eta_0$ small enough so that $V(y)\leq \eta_0$ for $y\in \R^k$ implies that $y\in B(\upsigma_i,\upmu_0/2)$ for some $i\in \{1,\cdots,q\}.$  In view of \eqref{eq:trou}, there exists therefore $i\in \{1,\cdots,q\}$ such that 
 \begin{equation}
 \label{gavrochul}
  \vert u(s_0)-\upsigma_i\vert \leq \frac{\upmu_0}{2}.
  \end{equation}
  On the other hand, by integration we have for any $s\in [a, a+1]$
  \begin{equation}
  \label{grosgrouin}
  \vert u(s)-u(s_0)\vert \leq \int_{[s, s_0]} \vert \dot u \vert \leq   \big (\int_{[s, s_0]} \vert \dot u \vert^2\big)^\frac12 \leq (2\eta_0)^{\frac{1}{2}}.
  \end{equation}
We impose additionally that $(2\eta_0)^{\frac{1}{2}}\leq \frac{\upmu_0}{2}$, so that combining \eqref{gavrochul} and \eqref{grosgrouin} we obtain, for any $s\in [a, a+1]$,
 \begin{equation}
 \label{outofkill}
\vert u(s)-\upsigma_i\vert \leq\upmu_0.
\end{equation}
 Since this relation holds for any $a\in I$ such that  $ [a, a+1]\subset I$, inequality \eqref{outofkill} actually holds for any $s \in I$, and the proof is complete.
 \qed
 
 \bigskip
\noindent
{\bf Proof of  Corollary \ref{interface}.}  We consider the covering of $\R$ given by $\R=\underset{ n \in \Z}\cup I_n$, where $I_n$ denotes the interval $[n, n+1]$. Consider the subset $\mathcal I$ of $\Z$ defined by
  $$\mathcal I=\{n \in \Z ,  \  \mathcal D(u)\cap I_n \neq\emptyset \}.$$
  In view of Lemma \ref{clearingout} we have $\int_{I_n} e(u)\ge \eta_0$ for any $n\in\mathcal I$, so that we deduce
  $$M_0\geq \int_{\R}e(u) \geq \sum_{n \in \mathcal I} \int_{I_n}e(u)\geq (\sharp \mathcal I ) \eta_0,$$
   so that $\sharp \mathcal I\leq \frac{M_0}{\eta_0}$. 
  For $n \in \mathcal I$  choose $x_n\in \mathcal D(u) \cap I_n $, so that $I_n\subset [x_n-1,x_n+1]$. We have 
  $$\mathcal D(u)\subset\bigcup_{n\in\mathcal I}I_n\subset\bigcup_{n\in\mathcal I}[x_n-1,x_n+1]\, .$$
  \qed

\section{A first upper bound for the velocity of the front set}
\label{upper}

In this section we provide a first estimate  concerning the speed of the front set. It relies on the clearing-out lemma as well as on some maximum principle for the quantity $|v_\eps - \sigma_i|^2.$ The main result of this section is the following  
 
\begin{proposition} 
\label{firstvel} 
Assume that $v_\eps$ is a solution to \eqref{glepsilon} verifying assumption $(H_0^\eps)$. Let $t\geq 0$, $x\in \R$ and $r\geq \alpha_0 \eps$ be such that 
\begin{equation}
\mathcal D(t) \cap [x-r,x+r] = \emptyset.
\end{equation}
Then, 
\begin{equation}\label{eq:klodenbis}
\mathcal D(s) \cap [x-\frac34 r,x+\frac34 r] = \emptyset
\end{equation}
for every $t\leq s \leq t+\alpha_0^{-3}r^2.$
\end{proposition}

The starting point is a straightforward consequence  of formula \eqref{eq:loca1}, often termed 
semi-decreasing property.

\begin{lemma}
\label{semidec}
Let $\chi$ be a smooth function with compact support in $\R,$ then
\begin{equation}\label{eq:vivachup}
\frac{1}{2}\int_{\R\times\{t\}}\eps |\partial_t v_\eps|^2\chi^2 + \frac{d}{dt} \int_{\R\times\{t\}} e_\eps(v_\eps) \chi^2 \leq 4 \|\dot  \chi\|_{L^\infty}^2 \int_{{\rm supp} \chi}e_\eps(v_\eps). 
\end{equation}
In particular,
\begin{equation}
\label{semidecroissance}
\frac{d}{dt}\int_{\R\times\{t\}} e_\eps(v_\eps) \chi^2(x)\,dx \leq 4\,
\|\dot \chi\|_{L^\infty}^2 \, \mathcal{E}_\eps(v_\eps^0) \leq 4M_0\|\dot \chi\|_{L^\infty}^2,
\end{equation}
so that for $0\leq t \leq t+\Delta t$
\begin{equation}
\label{semidecroissance2}
\int_{\R} e_\eps(v_\eps(x,t+\Delta t)) \chi^2(x)\,dx \leq \int_{\R} e_\eps(v_\eps(x,t)) \chi^2(x)\,dx + 4M_0\|\dot \chi\|_{L^\infty}^2\Delta t,
\end{equation}
whereas in the opposite direction
\begin{equation}\label{semidecroissance3}
\begin{split}
	\int_{\R} e_\eps(v_\eps(x,t+\Delta t)) \chi^2(x)\,dx \geq &\int_{\R} e_\eps(v_\eps(x,t)) \chi^2(x)\,dx\\ &- (\int_t^{t+\Delta t}\hspace{-5pt}\int_{\R}\eps |\partial_t v_\eps|^2\chi^2) - 4M_0\|\dot \chi\|_{L^\infty}^2\Delta t.
\end{split}
\end{equation}
\end{lemma}

This property yields

\begin{lemma} 
\label{guingamp}
Under the assumptions of Proposition \ref{firstvel}, there exist 
$$
x_i\in [x-r,x-\frac{3}{4}r] \qquad\text{and}\qquad x_f\in [x+\frac34 r,x+r]
$$
such that 
$$
\{x_i,x_f\} \cap \mathcal D (s) = \emptyset
$$
for every $t\leq s\leq t+\alpha_0^{-3}r^2.$ 
\end{lemma}
In other words, the two time segments $\{x_i\}\times [t,t+\alpha_0^{-3}r^2]$ and $\{x_f\}\times [t,t+\alpha_0^{-3}r^2]$ are disjoint from the space-time front set of $v_\eps.$ 
 
\begin{proof}
By the pigeonhole principle there exists  
$
x_i\in [x-r,x-\frac{3}{4}r] 
$
such that 
$$
\int_{x_i-\frac{2r}{\alpha_0}}^{x_i+\frac{2r}{\alpha_0}} e_\eps(v_\eps(y,t))\, dy \leq \frac{\eta_0}{2}.
$$
Indeed, the total energy of $v_\eps$ is bounded by $M_0$ and an interval of length $r/4$ is $\alpha_0/16$ times larger than an interval of length $4r/\alpha_0.$ 

Let $\chi$ be a smooth non-negative cut-off function such that $\chi\equiv 1$ on $[x_i-\eps,x_i+\eps]$, $\chi\equiv 0$ outside $[x-2r/\alpha_0,x_i+2r/\alpha_0]$, and $|\chi'| \leq \alpha_0 r^{-1}.$ Applying \eqref{semidecroissance2} we obtain
$$
\int_{x_i-\eps}^{x_i+\eps} e_\eps(v_\eps(y,s))\, dy \leq \frac{\eta_0}{2} + 4M_0 \alpha_0^{2} r^{-2}(s-t)\leq \eta_0
$$
provided
$$
s\leq t + \tfrac{\eta_0}{8M_0} \alpha_0^{-2} r^{2} \leq t + \alpha_0^{-3} r^2.
$$
One obtains similarly $x_f$ and the conclusion follows from Lemma \ref{clearingout}.
\end{proof}
 
By continuity, it follows from previous lemma that under the assumptions of Proposition \ref{firstvel} there exists $\sigma_i\in \Sigma$ such that $v_\eps$ takes values in $B(\sigma_i,\upmu_0)$ 
on the parabolic boundary 
$$
\partial_P \Lambda = [x_i,x_f]\times \{t\} \bigcup \{x_i\}\times [t,t+\alpha_0^{-3}r^2] \bigcup \{x_f\}\times  [t,t+\alpha_0^{-3}r^2]
$$
of the cylinder
$$
\Lambda = [x_i,x_f]\times [t,t+\alpha_0^{-3}r^2].
$$

We are now in position the complete the

\medskip

\noindent
{\bf Proof of Proposition \ref{firstvel}.} 
	Let $\tilde V$ be a potential which coincides with $V$ on $B(\sigma_i,\upmu_0)$ and such that the equivalent of \eqref{eq:conv} for $\tilde V$ holds everywhere on $\R^k.$ Consider the unique solution $\tilde v_\eps$ of 
	\begin{equation}
	\label{glepsilontilde}
 {\partial_t \tilde v_\eps}- \partial_{xx}{\tilde v_\eps}=-\frac{1}{\eps^2}\nabla \tilde V(\tilde v_\eps)
  \end{equation}
on $\Lambda$, which coincides with $v_\eps$ on $\partial_P \Lambda.$ By scalar multiplication of \eqref{glepsilontilde} with $\tilde v_\eps - \sigma_i$ and taking into account \eqref{eq:conv}, we are led to
  $$
  \partial_t |\tilde v_\eps-\sigma_i|^2 - \partial_{xx} |\tilde v_\eps-\sigma_i|^2 \leq - 2 |\partial_x (\tilde v_\eps-\sigma_i)|^2 \leq 0 
  $$
  on $\Lambda.$  It follows from the maximum principle that $|\tilde v_\eps-\sigma_i|^2$ achieves its maximal value on $\partial_P \Lambda$, and hence that $\tilde v_\eps \in B(\sigma_i,\upmu_0)$ on the whole $\Lambda.$ Since $V=\tilde V$ on $B(\sigma_i,\upmu_0)$, we deduce that $v_\eps = \tilde v_\eps$ on $\Lambda$, which is the desired conclusion. 
\qed

\section{Estimates off the front set}
\label{off}

In the previous section, we have proved Proposition \ref{firstvel} which corresponds to the first part of the claim (namely \eqref{eq:kloden}) in Proposition \ref{regularisation}. The purpose of this section is to present a result which will yield the second part of the claim, namely \eqref{eq:mouloud3}.

\begin{proposition}\label{estimpar}
	Let $v_\eps$ be a solution to \eqref{glepsilon} verifying assumption $(H_0)$.  Let $x\in \R$, $t\geq0$, $r>0$ and $s>t$ such that 
	$$
	v_\eps(y,\tau) \in B(\sigma_i,\upmu_0) \quad \text{for all }\ (y,\tau) \in [x-3r/4,x+3r/4]\times[t,s].
	$$
	Then 
	$$
	\int_{x-r/2}^{x+r/2} e_\eps(v_\eps(y,s))\, dy \leq M_0 (1+4\frac{\lambda_i^+}{\lambda_i^-}) \left[ \exp(-\frac{\lambda_i^-}{\eps^2}(s-t) + 2^{14}\frac{s-t}{r^2} \exp(-\sqrt{\frac{\lambda_i^-}{2}}\frac{1}{12}\frac{r}{\eps}) \right]. 
	$$
\end{proposition}

\begin{proof}
	Let $0\leq\varphi\leq 1$ be a smooth cut-off function such that\footnote{The constants hereafter were computed for a sine transition function.} $\varphi\equiv 0$ outside $[x-3r/4,x+3r/4]$, $\varphi\equiv 1$ on $[x-2r/3,x+2r/3]$, $|\varphi'|\leq 6\pi/r$ and $|\varphi''|\leq 72\pi^2/r^2.$  Let $w_\eps=\partial_x v_\eps,$ so that $w_\eps$ satisfies the equation
	\begin{equation}\label{eq:wiz}
		\partial_t w_\eps - \partial_{xx}w_\eps = -\frac{1}{\eps^2}\nabla^2V(v_\eps)w_\eps.
	\end{equation}
	We set\footnote{Notice that $f_\eps$ is implicitly defined on the whole $\R\times [t,s]$.} $f_\eps = \varphi w_\eps$, so that
\begin{equation}\label{eq:wiz2}
		\partial_t f_\eps - \partial_{xx}f_\eps + \frac{1}{\eps^2}\nabla^2V(v_\eps)f_\eps = -2\varphi'\partial_x w_\eps - \varphi'' w_\eps
	\end{equation}
	on $\R \times [t,s].$ By scalar multiplication with $2f_\eps$ we are led to
\begin{equation}\label{eq:wiz3}
		\partial_t |f_\eps|^2 - \partial_{xx}|f_\eps|^2 + \frac{2}{\eps^2}\nabla^2V(v_\eps)f_\eps \cdot f_\eps = -2|\partial_x f_\eps|^2 - 4\varphi \varphi' w_\eps \partial_x w_\eps  - 2\varphi''\varphi | w_\eps|^2.
	\end{equation}
Since $\varphi \partial_xw_\eps = -\varphi'w_\eps + \partial_x f_\eps,$ we have
$$
\varphi^2 |\partial_xw_\eps|^2 \leq 2 |\partial_xf_\eps|^2 + 2(\varphi')^2|w_\eps|^2,
$$
and since
$$
-4\varphi\varphi' w_\eps \partial_x w_\eps \leq \varphi^2|\partial_x w_\eps|^2 + 4(\varphi')^2|w_\eps|^2,
$$
we obtain
$$
\partial_t |f_\eps|^2 - \partial_{xx}|f_\eps|^2 + \frac{2}{\eps^2}\nabla^2V(v_\eps)f_\eps \cdot f_\eps \leq \left( 6(\varphi')^2 + 2|\varphi''|\varphi\right) |w_\eps|^2.
$$
In view of the pointwise bounds on $\varphi$ and its derivatives, as well as of \eqref{eq:conv}, we are led to
\begin{equation}\label{eq:evid}
	\partial_t |f_\eps|^2 - \partial_{xx}|f_\eps|^2 + \frac{2}{\eps^2}\nabla^2V(v_\eps)f_\eps \cdot f_\eps \leq 360\pi^2 r^{-2} |w_\eps|^21_A,
\end{equation}
where $A=[x-3r/4,x-2r/3]\cup [x+2r/3,x+3r/4]$ contains the support of $\varphi'.$ 

It follows from the comparison principle and Duhamel's formula that
\begin{equation}\begin{split}
|f_\eps|^2(\cdot,s) &\leq \exp(-\frac{\lambda_i^-}{\eps^2}(s-t)) G_{s-t} * |f_\eps|^2(\cdot,t)\\
&+ 360\pi^2r^{-2}\int_t^s \exp(-\frac{\lambda_i^-}{\eps^2}(s-\tau) )G_{s-\tau} * |w_\eps|^21_A(\cdot,\tau)\, d\tau,
\end{split}\end{equation}
where 
$$
G_\tau(y) = \frac{1}{\sqrt{4\pi\tau}}\exp(-\frac{y^2}{4\tau}).
$$
In particular, it follows from the triangle inequality that
\begin{equation}\label{eq:belette1}\begin{split}
	\|f_\eps(\cdot,s)\|_{L^2(B)}^2 &\leq \exp(-\frac{\lambda_i^-}{\eps^2}(s-t)) \|f_\eps(\cdot,t)\|^2_{L^2(\R)}\\
	&+ 360\pi^2r^{-2}\int_t^s \exp(-\frac{\lambda_i^-}{\eps^2}(s-\tau) )\|G_{s-\tau} * |w_\eps|^21_A(\cdot,\tau)\|_{L^1(B)}\, d\tau,
\end{split}\end{equation}
where $B=[x-r/2,x+r/2].$

Since $d(A,B)\geq r/6$, we have, for $t<\tau<s$,
\begin{equation}\label{eq:belette2}\begin{split}
\|G_{s-\tau} * |w_\eps|^21_A(\cdot,\tau)\|_{L^1(B)} & \leq |B| \|G_{s-\tau} * |w_\eps|^21_A(\cdot,\tau)\|_{L^\infty(B)}\\
&\leq |B| \max_{x\in B} \|G_{s-\tau}\|_{L^\infty(A)} \||w_\eps|^21_A(\cdot,\tau)\|_{L^1(B)}\\
&\leq r  \tfrac{1}{\sqrt{4\pi(s-\tau)}}\exp(-\tfrac{r^2}{144(s-\tau)}) 2\eps^{-1}M_0,
\end{split}
\end{equation}
where we have used the convolution inequality $||g*h||_\infty\le||g||_\infty||h||_1$ and assumption $(H_0).$ 

Combining \eqref{eq:belette1} with \eqref{eq:belette2} we obtain
\begin{equation}\label{eq:belette3}\begin{split}
	\|f_\eps(\cdot,s)\|_{L^2(B)}^2 \leq &2\eps^{-1}M_0\exp(-\frac{\lambda_i^-}{\eps^2}(s-t)) \\
	&+ 360\pi^{\frac32}r^{-2}\eps^{-1}M_0\int_t^s \exp\left(-\tfrac{\lambda_i^-}{\eps^2}(s-\tau)- \tfrac{r^2}{144(s-\tau)}\right)\tfrac{r}{\sqrt{s-\tau}} \, d\tau.
\end{split}
\end{equation}
We next use the inequality $x\leq \frac{12}{\sqrt{e}} \exp(x^2/288)$ in order to bound
$$
\exp\left(-\tfrac{r^2}{144(s-\tau)}\right)\tfrac{r}{\sqrt{s-\tau}} \leq \frac{12}{\sqrt{e}} \exp\left(-\tfrac{r^2}{288(s-\tau)}\right),
$$
and then the inequality
$$
\frac{\lambda_i^-}{\eps^2}(s-\tau) + \frac{r^2}{288(s-\tau)} \geq \sqrt{\frac{\lambda_i^-}{2}} \frac{r}{6\eps}
$$
in order to obtain
$$
\exp\left(-\tfrac{\lambda_i^-}{\eps^2}(s-\tau)-\tfrac{r^2}{144(s-\tau)}\right)\tfrac{r}{\sqrt{s-\tau}} \leq \exp(- \sqrt{\frac{\lambda_i^-}{2}}\frac{r}{6\eps}).
$$
Inserting the last inequality in \eqref{eq:belette3}, multiplying by $\frac{\eps}{2}$ and integrating yields therefore
\begin{equation}\label{eq:belette4}
	\int_{x-r/2}^{x+r/2} \eps\frac{|\partial_x v_\eps|^2}{2}(y,s)\, dy \leq M_0\left[  \exp\left(-\tfrac{\lambda_i^-}{\eps^2}(s-t)\right) + 2^{14}\frac{s-t}{r^2}  \exp\left(- \sqrt{\frac{\lambda_i^-}{2}} \frac{r}{6\eps}\right)\right].
\end{equation}
A completely similar computation for the function $g_\eps = \varphi (v_\eps-\sigma_i)$, in view of \eqref{eq:conv}
yields the estimate
\begin{equation}\label{eq:belette5}
	\int_{x-r/2}^{x+r/2} \frac{V(v_\eps)}{\eps}(y,s)\, dy \leq 4\frac{\lambda_i^+}{\lambda_i^-}M_0\left[  \exp\left(-\tfrac{\lambda_i}{\eps^2}(s-t)\right) + 2^{14}\frac{s-t}{r^2}  \exp\left(- \sqrt{\frac{\lambda_i}{2}}\frac{r}{6\eps}\right)\right].
	\end{equation}
The conclusion follows by summation.
\end{proof}

\medskip

\noindent{\bf Proof of Proposition \ref{regularisation} completed.} It is an immediate consequence of Proposition \ref{firstvel} and Proposition \ref{estimpar}.\qed

%
\section{ The notion of $\kappa$-confinement and optimal coverings}
\label{kappa}

\begin{definition}\label{def:conf}
	Let $X$ be a metric space and $S\subset X.$ Given a collection of distinct points $J$ in $X$, $\rho>0$ and $0<\kappa<1$, we say that $(J,\rho)$ is a $\kappa$-confined covering of $S$ at scale $\rho$ if
	\begin{equation}\label{eq:defconf1}
		S \subset \cup_{a\in J} B(a,\kappa \rho)\, ,\qquad S\cap B(a,\kappa\rho)\neq\emptyset\ \forall\, a\in J
	\end{equation}
	and
	\begin{equation}\label{eq:defconf2}
		{\rm dist}(a,b)\ge \kappa^{-1}\rho \qquad\forall a\neq b\ \text{in } J
	\end{equation}
\end{definition}
 
We  would like to draw the attention of the reader to the fact that Definition \ref{def:conf} involves in particular two parameters: the first one, $\rho$ is the typical length scale at which the confinement takes place, whereas the second one, $\kappa$, controls the rate of the confinement. 

A simple combinatorial argument yields

\begin{lemma}
\label{andouillette} 
Let $X$ be a metric space, and consider
 $\ell$ distinct points $a_1,...,a_\ell$ in $X$. Let $\delta>0$ and $0< \kappa < 1$ be given. Then there exists $\rho>0$ such that
 \begin{equation}
 \label{confin1}
 \delta\le \rho \le \kappa^{-2(\ell-1)} \delta
 \end{equation}
 and a subset $J$ of $\{a_i\}_{1\le i\le\ell}$ such that $(J,\rho)$ is a $\kappa$-confined covering of $S= \{a_1,\cdots,a_\ell\}$ at scale $\rho.$ 
 \end{lemma}
\begin{proof}
The proof is by iteration of merges in a finite number of steps. First,
consider the collection $J=\{a_1,\cdots,a_\ell\}$. Obviously, \eqref{eq:defconf1} is satisfied 
whatever our  choice of $\rho>0$ is. If \eqref{eq:defconf2}  is verified with the choice $\rho=\delta$ then there is nothing else to do. Otherwise, there are two points, say $a_1$ and $a_2$  after a possible relabelling, such that
\begin{equation}
\label{viva}
{\rm dist}(a_1,a_2)< \kappa^{-1}\delta.
\end{equation}
We then consider the collection $J=\{a_2,\ldots, a_\ell\}$ and set 
$\rho=\kappa^{-2}\delta$. Since by assumption \eqref{viva}
$$a_1\in B(a_2, \kappa\rho),$$ 
condition  \eqref{eq:defconf1} is therefore verified. As above, either \eqref{eq:defconf2} is verified for this choice of $\rho$ and $J$ or we go on in the same way. If the
process does not stop in $\ell-1$ steps, at the $(\ell-1)^{\text{\rm
th}}$ step we are left with one single element in $J$ and $\rho=\kappa^{-2(\ell-1)}\delta.$ At that point \eqref{eq:defconf2} becomes void and hence satisfied.
\end{proof} 

We now apply the previous result to the front set. An easy consequence of Corollary \ref{interface} and Lemma \ref{andouillette} is

\begin{lemma}\label{lem:conf}
Let $\eps>0$ and $u$ be such that
$$
\mathcal E_\eps(u) \leq M_0 < +\infty.
$$
Given $\delta >0$ and $0<\kappa<1$ verifying $\kappa\delta>2\eps,$ there exist  
$$
\delta\le \rho \le (\frac \kappa 2)^{-2M_0/\eta_0} \delta
$$
and a finite subset $J\subset \mathcal D(u)$ such that $\sharp(J)\leq M_0/\eta_0$ and $(J,\rho)$ is a 
$\kappa$-confined covering at scale $\rho$ of $\mathcal D(u).$ 
\end{lemma}
\begin{proof} By Corollary \ref{interface} and \eqref{rescaling} we deduce that there exist $\ell\le \frac{M_0}{\eta_0}$ and $\ell$ points $x_1,...,x_\ell$ such that $\mathcal D(u)\subset \cup_{i=1}^\ell[x_i-\eps,x_i+\eps]$. Lemma \ref{andouillette} used with $\frac{\kappa}{2}$ yields a $\frac{\kappa}{2}$-confined covering $(J,\rho)$ of $\{x_1,...,x_\ell\}$ at scale $\rho$, with $\delta\le\rho\le \delta (\frac{\kappa}{2})^{-2\frac{M_0}{\eta_0}}$ since $\sharp J\le \ell\le \frac{M_0}{\eta_0}$. Therefore, $\mathcal D (u)\subset \cup_{x_j\in J} [x_j-h,x_j+h]$, where $h=\frac{\kappa}{2}\rho+\eps\le \kappa\rho$. Since moreover $|x_j-x_k|\ge 2\kappa^{-1}\rho>\kappa^{-1}\rho$, $(J,\rho)$ yields a $\kappa$-confined covering at scale $\rho$ for $\mathcal D(u)$.
\end{proof}

\begin{definition}\label{def:opti}
	For $\eps$, $u$, $\delta$ and $\kappa$ as in Lemma \ref{lem:conf}, we set
	$$
	n(u,\kappa,\delta) = \inf \{ \sharp (J)\},   
	$$
	where the infimum ranges over the sets $J$ for which $(J,\rho)$ is a $\kappa$-confined covering of $\mathcal D(u)$ for some $\delta\leq \rho\leq (\frac\kappa 2)^{-2M_0/\eta_0}\delta.$  
\end{definition}

\section{Stopping times}
\label{concentrate}

In this  section we define  a notion of exit time for the front set, state and establish a result which is the core of the proof to Theorem \ref{maintheo}.

In the whole section, $v_\eps$ denotes a solution to \eqref{glepsilon} verifying $(H_0^\eps).$  
We fix a time $t\geq0$ and a length scale $\delta>0$. We also fix the value of $\kappa$ to
$$
\kappa_0 = 2\alpha_0^{-1}=\frac{\eta_0}{16M_0} \leq \frac{1}{16}. 
$$
We assume that $\kappa_0\delta>2\eps$, so that Lemma \ref{lem:conf} yields the existence of $\rho$ such that
$$
\delta\leq \rho \leq (\alpha_0)^{2 M_0/\eta_0} \delta
$$
and $J=\{a_1,\cdots,a_\ell\} \subset \mathcal D(t))$ such that $(J,\rho)$ is a $\kappa_0$-confined covering of $\mathcal D(t))$ at scale $\rho$. Without loss of generality, we may assume that this covering is optimal in the sense of Definition \ref{def:opti}, i.e. $\sharp (J)=n(v_\eps(\cdot , t),\kappa_0,\delta)$.

\subsection{Defining stopping and exit times}\label{sect:stop}

We define the exit time $T_1\in [t,+\infty]$ by
$$
T_1 \equiv T_1(v_\eps,t,(J,\rho)) = \inf \{ s\geq t \ \text{ s.t. }\ \mathcal
D(s) \not\subset \cup_{a\in J}B(a,\rho)\} ,
$$
and the dissipation time $T_2\in [t,+\infty]$ by
$$
T_2 \equiv T_2(v_\eps,t) = \inf \{ s\geq t \ \text{ s.t. }\ \eps\int_t^s\int_\R |\partial_\tau v_\eps|^2(x,\tau)\,dxd\tau \geq \frac{\eta_0}{8}\}.
$$
Finally, we define the target time $T_3$ by
$$
T_3\equiv T_3(t,\eps,\rho,M_0) =  t+ \rho^2 \frac{1}{2\sqrt{6K_V\alpha_0}}\exp(\frac{k_V}{2}\frac{\rho}{\eps}),
$$
where
$$
K_V = 2^{15}(1+4\max_{i=1}^q \frac{\lambda_i^+}{\lambda_i^-}),\qquad k_V = \min_{i=1}^q\{\min(\lambda_i^-, \sqrt{\frac{\lambda_i^-}{2}}\frac{1}{6})\}\, .
$$

\subsection{Dissipation or splitting}

Set
$
\beta_0 =4 {\alpha_0}^{\frac{\alpha_0}{16}}
$
and
$$
\gamma_0 = \max\{ \alpha_0\beta_0, \frac{\alpha_0^3}{k_V} \log (4\alpha_0^2K_V), \sqrt{6K_V\alpha_0/k_V}\}.
$$

The main result of this section is

\begin{proposition}\label{prop:core}
Assume that $\delta \geq \gamma_0\eps.$ If $T_1 < T_2$ and $T_1 < T_3$, then
\begin{equation}\label{eq:augmente}
	n(v_\eps(\cdot,T_1^-),\kappa_0,\frac{\delta}{\beta_0}) \geq
n(v_\eps(\cdot,t),\kappa_0,\delta)+1,
\end{equation}
where $T_1^- = T_1-4\kappa_0\alpha_0^{-3}\delta^2\geq t.$
\end{proposition}

In other words, if dissipation does not occur for a sufficiently long time, a
front may exit the confinement if and only if has already split before to yield an additional (well
separated at a finer scale) front. The small shift
in time in the definition of $T_1^-$ is motivated only by technical reasons.  

\medskip

\noindent{\bf Proof of Proposition \ref{prop:core}.} We divide the proof in a number of steps.

\smallskip

\noindent{\bf Step 1: lower bound on the exit time.} Set $t^+= t + 4\kappa_0^2 \alpha_0^{-3} \rho^2$. We have
$$
T_1^- \geq t^+\, ,\qquad\text{or equivalently, }\quad T_1\ge t+8\kappa_0^2\alpha_0^{-3}\rho^2\, .
$$
\begin{proof}
	By definition of $T_1$, it suffices to prove that for $t\leq s\leq t+  8\kappa_0^2 \alpha_0^{-3} \rho^2,$ $\mathcal D(s) \subset \cup_{a\in J} B(a,\rho).$ If $x\notin \cup_{a\in J} B(a,\rho),$ then since $\mathcal D(t) \subset \cup_{a\in J} B(a,\kappa_0\rho)$ and $\kappa_0\leq \frac{1}{16},$ we have $[x-4\kappa_0\rho,x+4\kappa_0\rho]\cap \mathcal D(t) = \emptyset.$ It follows from Proposition \ref{regularisation} that $x \notin \mathcal D(s)$ for $t\leq s \leq t+16\kappa_0^2 \alpha_0^{-3}\rho^2.$   
\end{proof}

\noindent{\bf Step 2: localising an exit point.} There exist $b_*\in \mathcal D(T_1)$ and $i_*\in \{1,\cdots,\ell\}$ such that
\begin{equation}\label{eq:proche}
	{\rm dist}(b_*,a_{i_*}) \in [\rho,(1+3\kappa_0)\rho]. 
\end{equation}
\begin{proof}
Let $x\in \R\setminus \cup_{a\in J}B(a,\rho+3\kappa_0\rho).$ Since $T_1^-<T_1$, it
follows that $[x-3\kappa_0\delta,x+3\kappa_0\delta]\cap \mathcal
D(T_1^-)=\emptyset.$ We induce from Proposition \ref{regularisation} that
$x\notin \mathcal D(s)$ for $t\leq s\leq T_1+5\kappa_0^2\alpha_0^{-3}\rho^2.$ By
definition of $T_1$, this implies that there exist a sequence of times $(t_j)_{j\in \N}$ in
$[T_1,T_1+5\kappa_0^2\alpha_0^{-3}\delta^2]$ such that $t_j$ decreases to $T_1$
as $j\to +\infty$ and a sequence of points
$(b_j)_{j\in \N}$  such that $b_j\in\mathcal D(t_j)$ and ${\rm dist}(b_j,\cup_{i=1}^\ell \{a_i\}) \in
[\rho,(1+3\kappa_0)\rho]$ for each $j\in \N.$  Let $b_*$ be an accumulation point of the sequence $(b_j)_{j\in \N}$.
Then $b_*\in\mathcal D(T_1)$ and there exists $i_*\in\{1,...,\ell\}$ such that dist$(a_{i_*},b_*)={\rm dist}(b_*\cup_{i=1}^\ell \{a_i\}) \in
[\rho,(1+3\kappa_0)\rho]$.
\end{proof}

\noindent{\bf Step 2bis: } There exist $c_*\in \mathcal D(T_1^-)$ such that
\begin{equation}\label{eq:proche2}
	{\rm dist}(c_*,b_*) \leq 2\kappa_0\rho.
\end{equation}
In particular,
\begin{equation}\label{eq:bienplace}
{\rm dist}(c_*,a_{i_*}) \in [\frac23\rho,\frac43\rho].
\end{equation}
\begin{proof}
Inequality \eqref{eq:proche2} follows directly from Proposition
\ref{regularisation} applied at time $T_1^-$ with $r=2\kappa_0\delta.$
Since $\kappa_0\leq 1/16,$ \eqref{eq:bienplace} then follows from
\eqref{eq:proche} and \eqref{eq:proche2}.
\end{proof}

\smallskip

\noindent{\bf Step 3: confinement of the energy after a boundary layer in time.} Let $x \in \R \setminus \cup_{a\in J} B(a,3\kappa_0\rho),$ then
	\begin{equation}\label{eq:petitesse}
		\int_{x-\kappa_0\rho}^{x+\kappa_0\rho} e_\eps(v_\eps(y,t^+))\, dy \leq K_V M_0 \exp(-k_V \alpha_0^{-3} \frac{\rho}{\eps}), 
	\end{equation}
	where $t^+ = t+4\kappa_0^2\alpha_0^{-3}\rho^2.$ 
	\begin{proof}
		Since $\mathcal D(t) \subset \cup_{a\in J} B(a,\kappa_0\rho),$ it follows that $[x-2\kappa_0\rho,x+2\kappa_0\rho]\cap \mathcal D(t) = \emptyset.$ Applying Proposition \ref{regularisation} with $r=2\kappa_0\rho$ we therefore obtain at $s=t^+$
		$$
		\int_{x-\kappa_0\rho}^{x+\kappa_0\rho} e_\eps(v_\eps(y,t^+))\, dy \leq M_0(1+4\frac{\lambda_j^+}{\lambda_j^-})\left[ \exp(-\lambda_j^-\alpha_0^{-3}\frac{4\kappa_0^2\rho^2}{\eps^2}) + 2^{14} \alpha_0^{-3} \exp(-\sqrt{\frac{\lambda_j^-}{2}}\frac{\kappa_0\rho}{3\eps})\right], 
		$$
		where $j\in\{1,\cdots,q\}$ is such that $v_\eps(x,t)\in B(\upsigma_j,\upmu_0).$ The conclusion follows from the definition of $K_V$ and $k_V$ noticing that since $\kappa_0\delta \geq 2\alpha_0\eps,$ we have
		$$
		\frac{4\kappa_0^2\rho^2}{\eps^2} \geq 4 \alpha_0 \frac{2\kappa_0\rho}{\eps}.
		$$
\end{proof}

\noindent{\bf Step 4: energy estimates off the front set.} For $t\leq s\leq T_1$ and for $i\in \{1,\cdots,\ell\}$ we have
\begin{equation}\label{eq:horscylindre}
	\int_{a_i+\frac{4}{3}\rho}^{a_i+2\rho}e_\eps(v_\eps(y,s))\, dy \leq K_V M_0 \left( \exp(-k_V \frac{s-t}{\eps^2}) + \frac{s-t}{\rho^2} \exp(-k_V \frac{\rho}{\eps})\right).  
\end{equation}
The same estimate holds for integration on the interval
$[a_i-2\rho,a_i-\frac{4}{3}\rho].$ 
\begin{proof}
	It suffices to apply Proposition \ref{estimpar} with $x=a_i+2\rho$ and $r=\frac{4}{3}\rho.$ Indeed, we have $[x-\frac{3}{4}r,x+\frac{3}{4}r] = [a_i+\rho,a_i+3\rho]$, and this set is disjoint from $\mathcal D(s)$ whenever $t\leq s\leq T_1$, by definition of $T_1.$ 
\end{proof}

\noindent
The next step contains the core of the argument.

\noindent{\bf Step 5: existence of a splitting.} There exists $c^* \in \mathcal
D(T_1^-)\cap B(a_{i_*},3\rho)$ such that 
\begin{equation}\label{eq:splitted}
	{\rm dist}(c_*,c^*) \geq \frac{\rho}{2}.
\end{equation}
\begin{proof}
Assume by contradiction that there are no such $c^*.$ We will obtain a
contradiction using the evolution equation \eqref{localizedenergy} for a suitable test function
$\chi.$ We  distinguish two cases.
If $c_*\geq a_{i_*}$, we have $c_*\geq a_{i_*}+\frac23\rho$ by
\eqref{eq:bienplace}, and this implies that
$$
\mathcal D(T_1^-) \cap [a_{i_*}-2\rho-2\kappa_0\rho, a_{i_*}+2\kappa_0\rho] =
\emptyset.
$$
We may thus apply Proposition \ref{regularisation} exactly as in the proof of Step 3, obtaining
\begin{equation}\label{eq:petitesse2}
		\int_{x-\kappa_0\rho}^{x+\kappa_0\rho} e_\eps(v_\eps(y,T_1))\, dy \leq K_V M_0 \exp(-k_V \alpha_0^{-3} \frac{\rho}{\eps}) 
	\end{equation}
	for every $x\in [a_{i_*}-2\rho,a_{i_*}]$.

In case $c_*< a_{i_*}$, we have therefore $c_*\leq a_{i_*}-\frac23\rho$ by
\eqref{eq:bienplace}, and this implies 
$$
\mathcal D(T_1^-) \cap [a_{i_*}-2\kappa_0\rho, a_{i_*}+2\rho+ 2\kappa_0\rho] =
\emptyset,
$$
so that we obtain as before
\begin{equation}\label{eq:petitesse3}
		\int_{x-\kappa_0\rho}^{x+\kappa_0\rho} e_\eps(v_\eps(y,T_1))\, dy \leq K_V M_0 \exp(-k_V \alpha_0^{-3} \frac{\rho}{\eps}) 
	\end{equation}
	for every $x\in [a_{i_*},a_{i_*}+2\rho].$

Since $b_*\in \mathcal D(T_1)$, we have by the clearing-out lemma,
	\begin{equation}\label{eq:co}
		\int_{b_*-\eps}^{b_*+\eps}e_\eps(v_\eps(y,T_1))\, dy \geq \eta_0.
	\end{equation}

Let $\chi$ be a smooth function with compact support in
$[a_{i_*}-2\rho,a_{i_*}+2\rho]$ such that $\chi(y) = y-a_{i_*}$ for $y\in
[a_{i_*}-\frac43\rho,a_{i_*}+\frac43\rho]$, $\chi(y)(y-a_{i_*}) \geq 0$ for
$y\in \R$, and\footnote{These last estimates are fulfilled for a (regularization
of a) degree three interpolation polynomial.} $\|\chi\|_\infty \leq 2\rho,$ $\|\chi''\|_\infty \leq 24\rho^{-1}.$

We integrate equality \eqref{localizedenergy} between $t^+$ and $T_1$ and
estimate each of its terms.

	Combining \eqref{eq:petitesse2} or \eqref{eq:petitesse3} with
\eqref{eq:co}, using the fact that by \eqref{eq:proche} $|a_{i_*}-y|\geq
\frac{1}{2}\rho $ for $y\in[b_*-\eps,b_*+\eps],$ and in view of the properties of $\chi,$ we are led to
	\begin{equation}\label{eq:estimhaut}
		|\int_\R e_\eps(v_\eps(y,T_1))\chi(y)\, dy| \geq
\frac{1}{2}\eta_0\rho - 2\kappa_0^{-1}K_V M_0 \exp(-k_V \alpha_0^{-3}
\frac{\rho}{\eps})\rho.
	\end{equation}

	We now estimate the same integral at time $t^+.$ Combining \eqref{eq:petitesse} outside $\cup_{a\in J}B(a,3\kappa_0\rho)$ with the bound $(H_0^\eps)$ inside  $\cup_{a\in J}B(a,3\kappa_0\rho)$, we are led to
\begin{equation}\label{eq:estimbas}
		|\int_\R e_\eps(v_\eps(y,t^+))\chi(y)\, dy| \leq
3\kappa_0M_0\rho + 2\kappa_0^{-1}K_V M_0 \exp(-k_V \alpha_0^{-3}
\frac{\rho}{\eps})\rho.
	\end{equation}

	We also have, since $\|\chi\|_\infty\leq 2\rho,$
\begin{equation}\label{eq:estimdissip}
		|\int_{t^+}^{T_1}\int_\R \eps |\partial_s
v_\eps(y,s)|^2\chi(y)\, dyds|\leq 2\rho\frac{\eta_0}{8} \leq \frac{1}{4}\eta_0\rho,
	\end{equation}
where we have used the fact that by assumption $T_1\leq T_2.$

	Finally, since $\|\chi''\|_\infty \leq 24\rho^{-1}$, since $\chi''$ is
supported in $[a_{i_*}-2\rho,a_{i_*}-\frac43\rho]\cup
[a_{i_*}+\frac43\rho,a_{i_*}+2\rho],$ and since
$|\xi_\eps(x)|\leq e_\eps$ pointwise, we obtain from \eqref{eq:horscylindre},
$$
|\int_\R  \xi_\eps(v_\eps(y,s))\chi''(y)\, dy| \leq 48\rho^{-1} K_VM_0 \left( \exp(-k_V \frac{s-t}{\eps^2}) + \frac{s-t}{\rho^2} \exp(-k_V \frac{\rho}{\eps})\right).  
$$
provided $t\leq s\leq T_1.$ Integrating the last inequality from $t^+$ to $T_1$
we are led to
\begin{equation}\label{eq:estimdiscrep}
|\int_{t^+}^{T_1}\int_\R  \xi_\eps(v_\eps(y,s))\chi''(y)\, dyds| \leq
48\rho^{-1} K_VM_0 \left( \frac{\eps^2}{k_V} + \frac{(T_1-t)^2}{2\rho^2} \exp(-k_V \frac{\rho}{\eps})\right). 	
	\end{equation}

Combining \eqref{eq:estimhaut}, \eqref{eq:estimbas}, \eqref{eq:estimdissip} and
\eqref{eq:estimdiscrep} with \eqref{localizedenergy} we deduce
\begin{multline}\label{eq:cle}
\big(\frac12\eta_0-3\kappa_0M_0-\frac14\eta_0\big)\rho - 4 \kappa_0^{-1}K_V M_0
\exp(-k_V \alpha_0^{-3} \frac{\rho}{\eps})\rho \\ \leq 48\rho^{-1} K_VM_0 \left( \frac{\eps^2}{k_V} + \frac{(T_1-t)^2}{2\rho^2} \exp(-k_V \frac{\rho}{\eps})\right). 	
\end{multline}

Since by assumption $\frac{\rho}{\eps} \geq \sqrt{6K_V\alpha_0/k_V},$ we have
\begin{equation}\label{eq:bouffe0}
48\rho^{-1} M_0\frac{\eps^2}{k_V} \leq \frac{\eta_0}{64}\rho.
\end{equation}
Since by assumption $\frac{\rho}{\eps} \geq \frac{\alpha_0^3}{k_V} \log (4\alpha_0^2K_V),$ we also have
\begin{equation}\label{eq:bouffe1}
4\kappa_0^{-1}K_V M_0 \exp(-k_V\alpha_0^{-3}\frac{\rho}{\eps})\rho \leq \frac{\eta_0}{64}\rho.
\end{equation}
Combining \eqref{eq:cle} with \eqref{eq:bouffe0} and \eqref{eq:bouffe1} we finally deduce
$$
48 \rho^{-1} K_V M_0 \frac{(T_1-t)^2}{2\rho^2} \exp(-k_V \frac{\rho}{\eps}) \geq \frac{\eta_0}{32}\rho,
$$
so that

$$
T_1 \geq t+ \rho^2 \frac{1}{2\sqrt{6K_V\alpha_0}}\exp(\frac{k_V}{2}\frac{\rho}{\eps})\geq T_3,
$$
the desired conclusion. 
\end{proof}

\noindent{\bf Step 6: persistence of the other fronts.}  For $i\in \{1,\cdots,\ell\}\setminus \{i_*\},$ 
$$
\mathcal D(T_1^-) \cap [a_i- 2\rho,a_i+2\rho] \neq \emptyset.
$$
\begin{proof}
The proof is by contradiction, very similar and actually simpler than the one of Step 5, so that we only briefly sketch it. It relies also on equation \eqref{localizedenergy}, with a function $\chi$ which is here taken to be non negative, identically equal to $1$ in $[a_i-\frac{4}{3}\rho,a_i+\frac{4}{3}\rho],$ with compact support in $[a_i-2\rho,a_i+2\rho]$ and verifying $\|\chi\|_\infty\leq 1$ and $\|\chi''\|_\infty\leq 14\delta^{-2}$.This time, the dominant term is given by $\int e_\eps \chi$ taken at the initial time $t$, and not at the exit time.  
\end{proof}

\noindent{\bf Step 7: proof of \eqref{eq:augmente}.} In view of Step 5 and Step 6, there exists $\ell+1$ points in $\mathcal D(T_1^-)$ such that the mutual distance between any two of them is bounded from below by $\frac{\rho}{2}.$ Any covering of $\mathcal D(T_1^-)$ with balls of radius smaller than $\rho/4$ must therefore contain at least $\ell+1$ balls. Notice that the existence of a $\kappa_0$-confined covering of $\mathcal D(T_1^-)$ at scale $\rho'\le\rho/4$ follows by  Lemma \ref{lem:conf} with $\delta$ replaced by $\delta'=\delta/\beta_0$: by assumption we have  $\kappa_0\delta'\ge 2\eps$ and the scale $\rho'$ of the $\kappa_0$-confined covering does not exceed $\delta'(\frac{\kappa_0}{2}^{-2M_0/\eta_0})\le \delta/4\le\rho/4$.
The conclusion follows. \qed

\section{Proof of Theorem 1bis
}
\label{ail2}

 We distinguish between two cases, depending on the value of the ratio $\frac{R}{\eps}.$

\medskip

\noindent{\bf Case I : $R$ is sufficiently large\footnote{The actual limiting ratio is given by \eqref{eq:pastroppetit}.} with respect to $\eps.$} 

This is the most interesting case, where we apply the analysis of the previous sections. In Section \ref{concentrate}, the initial time $t$ as well as the typical length scale $\delta$ were fixed. In order to prove Theorem 1bis
in the present case, we will use Proposition \ref{prop:core} inside an iteration argument, containing a finite and bounded from above number of steps, each step corresponding to a new time $t_n$ and a new length scale $\delta_n.$ 

 More precisely, we set $t_0=0$ and will fix the value of $\delta_0$ at the end of the argument. It remains to define the times $t_n$ and the scales $\delta_n$ iteratively.

 Let $n\geq 0$ be fixed. In view of Section \ref{concentrate}, provided $\delta_n \geq \gamma_0 \eps$ there exist
 \begin{equation}\label{eq:repete}\delta_n\leq \rho_n\leq (\frac{\kappa_0}{2})^{-2M_0/\eta_0}\delta_n\end{equation} and $J_n=\{a_1^n,\cdots,a_{\ell_n}^n\}$ such that $(J_n,\rho_n)$ is an optimal $\kappa_0$-covering of $\mathcal D(t_n)$ at scale $\rho_n.$ 
Define the corresponding stopping and exit times $T_{1,n},T_{2,n},T_{3,n}$ according to Section \ref{sect:stop}.
\medskip

We distinguish three cases : 

\medskip

{\bf \mathversion{bold} $({\rm Case \ 1)}_n  :\ T_{3,n} < T_{1,n}$  \mathversion{normal}} In that case, we define $t_{n+1}\equiv t_{\rm fin} = T_{3,n}$ and the iteration process stops. 

\medskip

{\bf \mathversion{bold} $({\rm Case \ 2)}_n  :\  T_{3,n} \geq T_{1,n} \text{ and } T_{2,n}< T_{1,n}$          \mathversion{normal}} In that case, corresponding to a dissipation time, we set $t_{n+1}= T_{2,n}$ and $\delta_{n+1}=\delta_0$ (we reset the resolution scale to its initial value). 

\medskip

{\bf \mathversion{bold} $({\rm Case \ 3)}_n : \   T_{1,n} \leq \min(T_{2,n},T_{3,n}) $          \mathversion{normal}} In that case, corresponding to a splitting, we set $t_{n+1} = T_{1,n}^-$ and $\delta_{n+1}= \frac{1}{\beta_0}\delta_n.$ 

\medskip

By construction, in each case,
\begin{equation}\label{eq:contenu}
	\mathcal D(s) \subset \mathcal D(t_n) + [-2\rho_n,2\rho_n] \qquad \text{ for } t_n\leq s \leq t_{n+1}.
\end{equation}

Moreover, if $({\rm Case \ 3)}_n$ holds, then by Proposition \ref{prop:core}
\begin{equation}\label{eq:splitencore}
n(v_\eps(\cdot,t_{n+1}),\kappa_0,\delta_{n+1}) \geq n(v_\eps(\cdot,t_{n}),\kappa_0,\delta_{n}) +1.
\end{equation}

Since $\delta_{n+1}\geq \eps$ by assumption, $n(v_\eps(\cdot,t_{n+1}),\kappa_0,\delta_{n+1})\leq \frac{M_0}{\eta_0}$ by Corollary \ref{interface}.
It follows by \eqref{eq:splitencore} that $({\rm Case \ 3)}_n$ may only occur for at most $\frac{M_0}{\eta_0}$ consecutive values of $n$, provided $\delta_n$ remains greater than $\gamma_0\eps.$ In view of the iteration process, this also implies that $\delta_n$ is bounded from below by
\begin{equation}\label{eq:bornebas}
	\delta_n \geq \beta_0^{-\frac{M_0}{\eta_0}}\delta_0.
\end{equation}

On the other hand, since the total dissipation is bounded from above by the total energy of the initial datum, $({\rm Case \ 2)}_n$, in view of the definition of $T_{2,n}$, may only occur for at most $4\frac{M_0}{\eta_0}$ distinct values of $n.$ 

We induce from the two previous limitations that the process may contain at most 
$$
(\frac{M_0}{\eta_0}+1)({4\frac{M_0}{\eta_0}}) +1
$$
steps, so that $t_{\rm fin}$ is well defined. Moreover by \eqref{eq:contenu}, \eqref{eq:repete}, the fact that $\sum_{k=0}^\infty \beta_0^{-k}\leq 2,$ and the triangular inequality,  we obtain
\begin{equation*}
	\mathcal D(s) \subset \mathcal D(0) + [-r_{\rm fin}, r_{\rm fin}] \qquad \text{ for } 0\leq s \leq t_{fin},
\end{equation*}
where
$$
r_{\rm fin} = 4(\frac{\kappa_0}{2})^{-2\frac{M_0}{\eta_0}}\left( 4\frac{M_0}{\eta_0}+1\right)\delta_0 =\alpha_0^{\frac{\alpha_0}{16}}(\frac{\alpha_0}{2}+4)\delta_0. 
$$
We may now define the value of $\delta_0$ to be
$$
\delta_0 = R{\alpha_0^{-\frac{\alpha_0}{16}}(\frac{\alpha_0}{2}+4)^{-1}}.
$$
If
\begin{equation}\label{eq:pastroppetit}
R \geq (\alpha_0)^{\frac{\alpha_0}{16}}(\frac{\alpha_0}{2}+4) \beta_0^{\frac{M_0}{\eta_0}} \gamma_0\eps,
\end{equation}
we infer from \eqref{eq:bornebas} that the condition $\delta_n\geq \gamma_0\eps$ is met along the whole process. We also infer from \eqref{eq:bornebas} and the definition of $T_{3,n}$ that
$$
t_{\rm fin} \geq \left(\frac{R}{K_0}\right)^2 \exp(\frac{1}{K_0}\frac{R}{\eps}) 
$$
where
where
\begin{equation*}\begin{split}
K_0 &=\beta_0^{\frac{M_0}{\eta_0}}
\alpha_0^{\frac{\alpha_0}{16}}(\frac{\alpha_0}{2}+4)\max\{\root{4}\of{24 K_V\alpha_0},\,  \frac{2}{k_V}\} \\
&\geq 2^{10}\left(32\frac{M_0}{\eta_0}\right)^{2(\frac{M_0}{\eta_0})(1+\frac{M_0}{\eta_0})}\max_{i=1}^q\{ \max(\root{4}\of{\frac{\lambda_i^+}{\lambda_i^-}},\frac{1}{\lambda_i^-})\}.
\end{split}\end{equation*}

The proof is then completed in case \eqref{eq:pastroppetit} holds.

\medskip

\noindent{\bf Case II : \eqref{eq:pastroppetit} does not hold but $R\geq \alpha_0\eps$.}

In that case, it suffices to invoke Proposition \ref{regularisation}. The conclusion follows adjusting the constant $K_0$ if necessary.\qed

\section{Proof of Theorem \ref{thm:pouf}}

It is a direct consequence of Theorem 1bis and Proposition \ref{estimpar} that for $0\leq t\leq (\frac{R}{K_0})^2\exp(\frac{R}{K_0\eps})$ and $x\in [x_0-\frac12 R, x_0+\frac12R],$  
\begin{equation}\label{eq:interetape}
	\int_{x-\eps}^{x+\eps} e_\eps(v_\eps(y,t)\, dy \leq K M_0 \left[ \exp\left(-\frac{t}{K\eps^2}\right) + \frac{t}{R^2}\exp\left(-\frac{R}{K\eps}\right)  \right],
\end{equation}
for some constant $K$ depending only on $V.$ To derive the pointwise bounds given by \eqref{eq:plouf1}, it suffices then to invoke parabolic regularization and scaling. Indeed, if $v$ is a solution of \eqref{glpara} on a cylinder of the type $[y-1,y+1]\times[s-1,s]$, such that $v \in B(\upsigma_i,\upmu_0)$ for some $i\in \{1,\cdots,q\}$ on that cylinder, then 
$$
|\partial_t v(y,s)|^2 + e(v(y,s)) \leq K' \int_{s-1}^s \int_{y-1}^{y+1} e(v(z,\tau))\, dzd\tau.
$$
Therefore, by scaling if $v_\eps$ is a solution of \eqref{glepsilon} on a cylinder of the type $[y-\eps,y+\eps]\times[s-\eps^2,s]$, such that $v_\eps \in B(\upsigma_i,\upmu_0)$ for some $i\in \{1,\cdots,q\}$ on that cylinder, then\begin{equation}\label{eq:scaledeps}
	\eps^4|\partial_t v_\eps(y,s)|^2 + \eps e_\eps(v_\eps(y,s)) \leq K' \eps^{-2}\int_{s-\eps^2}^s \int_{y-\eps}^{y+\eps} e_\eps(v_\eps(z,\tau))\, dzd\tau.
\end{equation}
The conclusion \eqref{eq:plouf1} with $K_1=2K\max(K',1)$ follows combining \eqref{eq:interetape} and \eqref{eq:scaledeps}.\qed

 \section{Relaxation to stationary fronts}
  The aim of this section is to provide a proof to Theorem \ref{theglue}. The starting idea is to determine a good time slice for which the   integral  of the dissipation $\vert \partial_t v_\eps \vert^2$ is small.   Then, the main part of the proof is devoted to the study of solutions to  the  perturbed ordinary differential equation
  \begin{equation}
   \label{odepertu}
   u_{xx}=\eps^{-2}\nabla V(u)+f \ \ {\rm  on } \  \R, 
  \end{equation}
   where the function $f$ belongs to $L^2(\R)$.
   
  \subsection{Study of the perturbed equation \eqref{odepertu} : initial value problem}
  It is useful to recast equation \eqref{odepertu} as a system of two differential equations of first order. For that purpose, we set 
  $w=\eps u_x$
   so that  \eqref{odepertu} is equivalent to the system
   \begin{equation}
   \label{prems}
   \begin{split}
   u_x&=\frac{1}{\eps}w \\
   w_x&=\frac{1}{\eps}\nabla V(u)+\eps f, 
      \end{split}
          \end{equation}
which we may write in the  condensed form 
\begin{equation}
\label{eqvect}
U_x=\frac{1}{\eps} G(U)+\eps F \ {\rm  on } \  \R, 
\end{equation}
where, for $x$ in $\R$, we have set  $U(x)=(u(x), w(x))$ and $F(x)=(0, f(x))$, and where $G$ denotes the vector field on $\R^{2k}$ given by $G(u_1,u_2)=(u_2, \nabla V(u_1))$. Notice that
$$\vert \nabla G(u_1, u_2) \vert \leq  A(\vert u_1 \vert),$$
 where $A\geq 1$ is  some continuous non-decreasing scalar function.
We next compare a given global bounded solution $u$ of \eqref{odepertu} to a solution  $u^0$ of the  unperturbed equation
  \begin{equation}
   \label{odenonpertu}
   u_{xx}=\eps^{-2}\nabla V(u) \ , 
  \end{equation}
with similar initial condition at some point $x_0\in \R$. 
We denote accordingly $U^0=(u^0, \eps^{-1} u_x^0)$ on its maximal interval of existence. 

\medskip

As a consequence of Gronwall's identity,  we have
  
 \begin{lemma} 
 \label{gron}
Let $u$ and $u_0$ be as above.  Assume that for some $x_0\in \R$ and $a>0,$  
\begin{equation}
\label{thesuppose}
|U(0)-U^0(0)| + \frac{\eps^\frac32}{\sqrt{2A(\|u\|_\infty +1)}} \|f\|_{2} \leq \exp\left( -\frac{A(\|u\|_\infty +1)a}{\eps}\right).
\end{equation}
Then $u^0$ is well defined on $[x_0-a,x_0+a]$ and we have
\begin{equation}
\label{loco}
\Vert U-U^0 \Vert_{L^\infty([x_0-a,x_0+a])} \leq \left( |U(0)-U^0(0)| + \frac{\eps^\frac32}{\sqrt{2A(\|u\|_\infty +1)}} \|f\|_{2}\right) \exp\left( \frac{A(\|u\|_\infty +1)a}{\eps}\right).
   \end{equation}
\end{lemma}
 
 \begin{proof}  Without loss of generality, we may assume that $x_0=0.$ Let $I$ be the largest interval containing $0$ and such that
 \begin{equation}
 \label{helicopter}
 \Vert u^0 \Vert_{L^\infty(I)} \leq \Vert u  \Vert_{\infty}+1.
 \end{equation}
On $I$, since $(U-U^0)_x= G(U)-G(U^0)+\eps F$ we obtain the inequality
$$\vert (U-U^0)_x\vert \leq \frac{ A \left(\Vert u  \Vert_{\infty} +1\right)}{\eps} \vert  U-U^0\vert +\eps \vert F \vert.$$
It follows from Gronwall's inequality, that, for $x \in I$, 
\begin{multline*}
\vert (U-U^0)(x)\vert \leq
   \exp \left( \frac{ A \left(\Vert u  \Vert_\infty +1\right)|x|}{\eps}\right)\vert (U-U^0)(0)\vert\\ + 
 | \int_0^x \eps \vert F(x-y)\vert  \exp \left( \frac{ A \left(\Vert u  \Vert_\infty +1\right)  \vert x\vert }{\eps}  \right) dy|,
  \end{multline*}
   so that by the Cauchy-Schwarz inequality, we are led to the bound, for $x\in I$,
  \begin{equation}
  \label{massa}
\Vert (U-U^0)(x) \Vert \leq \left( |U(0)-U^0(0)| + \frac{\eps^\frac32}{\sqrt{2A(\|u\|_\infty +1)}} \|f\|_{2}\right) \exp\left( \frac{A(\|u\|_\infty +1)|x|}{\eps}\right).
\end{equation}   
Hence, if \eqref{thesuppose} is verified, then $[-a,a] \subset I$ and \eqref{loco} follows.
\end{proof}
   
We will combine the previous lemma with

\begin{lemma}\label{lem:discrepok}
	Let $u$ be a global solution of \eqref{odepertu} such that $\mathcal{E}_\eps(u)\leq M_0<+\infty.$ Then
$$
\|\xi_\eps(u)\|_\infty \leq \sqrt{2} \eps M_0 \|f\|_2. 
$$
\end{lemma}
\begin{proof}
This is a direct consequence of the equality
$$
\frac{d}{dx} \xi_\eps(u) = \eps f \frac{d}{dx}u,
$$
Cauchy-Schwarz inequality, and the fact that $\xi_\eps(u)$ tends to zero at infinity since $u$ has finite energy.  
\end{proof}

\begin{corollary}\label{cor:9.1}
	Let $u$ be a global solution of \eqref{odepertu} such that $\mathcal{E}_\eps(u)\leq M_0<+\infty.$ There exist a constant $C_0>0$ depending only on $M_0$ and $V$ such that if $x_0\in \mathcal D(u)$ and if
\begin{equation}\label{eq:serend}
	b = \frac{\eps}{A(\|u\|_\infty+1)} \log\left(\frac{1}{C_0\eps^\frac32 \|f\|_2}\right) > 0,
\end{equation}
then there exists a solution $u^0$ of \eqref{odenonpertu} defined on $[x_0-b,x_0+b]$ and verifying 
\begin{eqnarray}
&&\xi_\eps(u^0) \equiv 0,\\
&&\|U-U^0\|_{L^\infty([x_0-a,x_0+a])} \leq  \exp\left( -\frac{A(\|u\|_\infty +1)}{\eps}(b-a)\right),
\end{eqnarray}
for every $0<a<b.$
\end{corollary}
\begin{proof}
Since $x_0\in \mathcal D(u),$ there exist a constant $c_0>0$ (depending only on the choice of $\upmu_0$ and the eigenvalues $\lambda_i^-$) such that
$$
\eps |u_x(x_0)|^2/2 = \frac{V(u(x_0)}{\eps} + \xi_\eps(u)(x_0) \geq \frac{c_0}{\eps} + \xi_\eps(u)(x_0) 
$$
so that by Lemma \ref{lem:discrepok},
$$
\eps|u_x(x_0)|^2/2 \geq \frac{c_0}{\eps} -\sqrt{2}\eps M_0 \|f\|_2.
$$
Since $a>0$, we infer that $C_0\eps^\frac32 \|f\|_2< 1$ and therefore
$$
\eps|u_x(x_0)|^2/2 \geq \frac{c_0}{\eps} -\sqrt{2}\eps M_0 C_0^{-1}\eps^{-\frac32}.
$$
We first require $C_0 \geq 2\sqrt{2}M_0/c_0$ so that we obtain, since $0<\eps\leq 1,$ 
\begin{equation}\label{eq:derivgrande}
	\eps|u_x(x_0)| \geq \sqrt{c_0}. 
\end{equation}
We let $u^0(x_0)=u(x_0)$ and we wish to define $u^0_x(x_0)$ in such a way that $\xi_\eps(u^0)(x_0)=0.$ This may be achieved in general in a non unique way. We choose $u^0_x(x_0)$ as the unique positive multiple of $u_x(x_0).$ It follows from the equality
$$
\eps|u_x(x_0)|^2 - \eps|u_x^0(x_0)|^2 = 2\xi_\eps(u)(x_0),
$$
from the bound \eqref{eq:derivgrande}, and from Lemma \ref{lem:discrepok}, that
$$
\left|\eps \left( u_x-u^0_x\right)(x_0)\right| \leq 2\sqrt{2}\frac{M_0}{\sqrt{c_0}}\eps^2\|f\|_2, 
$$
or, since $u(x_0)=u^0(x_0),$ that
\begin{equation}\label{eq:depprox}
\left|U(x_0)-U^0(x_0)\right| \leq  2\sqrt{2}\frac{M_0}{\sqrt{c_0}}\eps^2\|f\|_2.
\end{equation}
In order to apply Lemma \ref{gron}, we estimate
$$
|U(0)-U^0(0)| + \frac{\eps^\frac32}{\sqrt{2A(\|u\|_\infty +1)}} \|f\|_{2} \leq \left( 2\sqrt{2}\frac{M_0}{\sqrt{c_0}} + \frac{1}{2}\right) \eps^\frac32 \|f\|_2.
$$
We next require that $C_0 \geq 2\sqrt{2}\frac{M_0}{\sqrt{c_0}} + \frac{1}{2}$ so that by definition of $b$,
$$
|U(0)-U^0(0)| + \frac{\eps^\frac32}{\sqrt{2A(\|u\|_\infty +1)}} \|f\|_{2} \leq 
\exp\left( -\frac{A(\|u\|_\infty +1)b}{\eps}\right).
$$
The conclusion then follows from Lemma \ref{gron}.
\end{proof}

\subsection{Study of the perturbed equation \eqref{odepertu} : boundary value problem}

In this subsection, which may be viewed as the elliptic counterpart of Section \ref{off}, we prove estimates regarding a solution $u$ of \eqref{odepertu} on some interval, provided that this interval has an empty intersection with the front set of $u.$ 

More precisely, we have

\begin{lemma}\label{lem:elliptic}
	Let $u$ be a solution to \eqref{odepertu} such that $\mathcal E_\eps(u)\leq M_0<+\infty,$ and let $x_0\in \R,$ $r>0$ be such that 
	$$
[x_0-r,x_0+r] \cap \mathcal D(u) = \emptyset. 
	$$
	There exist a constant $C_1>0$ depending only on $M_0$ and $V$ such that
	\begin{equation}\label{eq:reguelli}
		\eps e_\eps(u) \leq C_1\left( \eps^\frac32\|f\|_2 + \frac{\eps}{r}\exp(-\frac{r}{C_1\eps})\right)
	\end{equation}
	on $[x_0-r/2,x_0+r/2].$ 
\end{lemma}
\begin{proof}
	The proof is very similar to the one of Proposition \ref{estimpar}, so that we skip part of the details. In the sequel, $C$ denotes a constant which depends only on $V$ and whose actual value may vary from line to line. If $i\in \{1,\cdots,q\}$ is such that $u(x) \in B(\upsigma_i,\upmu_0)$ on $[x_0-r,x_0+r],$ then the function $w$ defined by
	$$
w(x) = (u-\upsigma_i)^2\chi^2,
	$$
	where $0\leq \chi\leq 1$ is a smooth cut-off function equal to one on $[x_0-\frac34 r,x_0+\frac34 r]$ and with compact support in $(x_0-r,x_0+r)$, satisfy on $\R$ the differential inequality
	$$
	- w_{xx} + \frac{\lambda_i^-}{2\eps^2}w \leq C f + C r^{-2}1_{{\rm supp}(\chi')}. 
	$$
By the comparison principle, we obtain 
	$$
	w \leq C(K *f+ r^{-2}K*1_{{\rm supp}(\chi')}),
	$$
	where 
	$$
	K(x) = \frac{\eps}{\sqrt{2\lambda_i^-}}\exp\left(-\frac{\sqrt{\lambda_i^-/2}}{\eps}|x|\right).
	$$
We then estimate
$$
\|K*f\|_\infty \leq \|K\|_2\|f\|_2 \leq C\eps^\frac32 \|f\|_2,
$$
and since ${\rm dist}([x_0-r/2,x_0+r/2],{\rm supp}(\chi'))\geq r/4,$ 
\begin{equation*}\begin{split}
\|r^{-2}K*1_{{\rm supp}(\chi')}\|_{L^\infty([x_0-\frac{r}{2},x_0+\frac{r}{2}])} &\leq r^{-2}\|K\|_{L^\infty(\R\setminus [-\frac{r}{4},\frac{r}{4}])} {\rm meas}({\rm supp}(\chi'))\\
&\leq C \frac{\eps}{r}\exp\left(-\sqrt{\tfrac{\lambda_i^-}{32}}\frac{r}{\eps}\right).
\end{split}\end{equation*}
Since $\chi\equiv 1$ on $[x_0-r/2,x_0+r/2]$, we therefore obtain, for $x\in [x_0-r/2,x_0+r/2]$, 
\begin{equation}\label{eq:etdeun}
V(u(x)) \leq C\left( \eps^\frac32\|f\|_2 + \frac{\eps}{r}\exp(-\frac{r}{C\eps})\right).
\end{equation}
Finally, by definition of $\xi_\eps$ we have
$$
\eps^2 \frac{|u_x|^2}{2} = \eps \xi_\eps(u) + V(u)   
$$
so that for $x\in [x_0-r/2,x_0+r/2]$, by \eqref{eq:etdeun} and Lemma \ref{lem:discrepok},
\begin{equation}\label{eq:etdedeux}
	\eps^2|u_x|^2(x) \leq \sqrt{2}M_0\eps^2\|f\|_2 + C\left( \eps^\frac32\|f\|_2 + \frac{\eps}{r}\exp(-\frac{r}{C_1\eps})\right).
\end{equation}
Combining \eqref{eq:etdeun}, \eqref{eq:etdedeux} and the fact that $\eps\leq 1$, the conclusion follows for the choice $C_1= C+ \sqrt{2}M_0.$  
\end{proof}

\subsection{Proof of Theorem \ref{theglue}}

Since the total dissipation of energy is bounded from above by $M_0$, by averaging, there exist 
$$
0 \leq T \leq (\frac{R}{K_0})^2 \exp\left(\frac{R}{K_0\eps}\right)
$$
such that
$$
\eps \int_\R |\partial_t v_\eps(x,T)|^2 \, dx \leq K_0^2M_0R^{-2}\exp\left(-\frac{R}{K_0\eps}\right).
$$
Hence, $v_\eps(\cdot,T)$ satisfies equation \eqref{odepertu} where
$$
\|f\|_2 \leq K_0\sqrt{M_0}R^{-1}\eps^{-\frac12}  \exp\left(-\frac{R}{2K_0\eps}\right).
$$
In order to apply Corollary \ref{cor:9.1}, and in view of \eqref{eq:serend}, we first estimate
\begin{equation}\label{eq:montre}\begin{split}
b &\equiv  \frac{\eps}{A(\|v_\eps(\cdot,T)\|_\infty+1)} \log\left(\frac{1}{C_0\eps^\frac32 \|f\|_2}\right) \\
&\geq  \frac{\eps}{A(\|v_\eps(\cdot,T)\|_\infty+1)} \log\left(\frac{1}{C_0K_0\sqrt{M_0}} \frac{R}{\eps} \exp\left(\frac{R}{2K_0\eps}\right)\right) \\
& \geq  2\frac{R}{K_2},
\end{split}\end{equation}
provided $K_2$ is large enough depending only on $M_0$ and $V.$ Next, we apply Lemma \ref{lem:conf} with $\kappa=\frac{1}{4}$ and $\delta = \frac{R}{K_2}2^{-4\frac{M_0}{\eta_0}}.$ This yields  
$$
\frac{R}{K_2}2^{-4\frac{M_0}{\eta_0}} \leq r \leq \frac{R}{K_2}
$$
and a collection of points $\{a_j\}_{j\in J}$ such that $\sharp J \leq M_0/\eta_0,$  
$$
\mathcal D(T) \subset \cup_{j\in J} (a_j-r/4,a_j+r/4)
$$
and
$$
{\rm dist}(a_i,a_j) \geq 4r \qquad \forall\ i\neq j\in J.
$$
In particular, $2r\leq b$ and for each $j\in J$ we may apply Corollary \ref{cor:9.1} with $x_0=a_j$ and $a=r.$ This yields, after rescaling, the collection $\{U_j\}_{j\in J}$ and the estimates
\begin{equation*}\begin{split}
\|v_\eps(\cdot,T)-U_j(\frac{\cdot-a_j}{\eps})\| + \eps\|\partial_x\left(v_\eps(\cdot,T)-U_j(\frac{\cdot-a_j}{\eps})\right) \| &\leq \exp\left( - A(\|v_\eps\|_\infty+1) \frac{r}{\eps}\right)\\
&\leq K_2\exp\left(- \frac{r}{K_2\eps}\right)
\end{split}\end{equation*}
in the space $L^\infty([a_j-r,a_j+r]),$ provided $K_2$ is large enough, depending only on $M_0$ and $V.$  

Consider now $x_0\in \R$ such that ${\rm dist}(x,\cup_{j\in J}\{a_j\})\geq r.$ Since $[x_0-\frac34r,x_0+\frac34r]\cap \mathcal D(T) =\emptyset$ by construction, we obtain by Lemma \ref{lem:elliptic}
$$
\eps e_\eps(v_\eps(x_0, T)) \leq C_1\left( K_0\sqrt{M_0}\frac{\eps}{R} \exp\left(-\frac{R}{2K_0\eps}\right) + \frac{4\eps}{3r}\exp\left(-\frac{3r}{4C_1\eps}\right)\right),
$$
so that
$$
\|v_\eps(x_0,T)-\upsigma_j\| + \eps\|\partial_x(v_\eps(x_0,T)\| \leq K_2 \exp\left( - \frac{r}{K_2\eps}\right),
$$
provided $K_2$ is large enough, depending only on $M_0$ and $V.$  The conclusion follows.\qed

\end{document}